\documentclass[10pt]{article}    

\usepackage{graphicx}          

\usepackage{amsthm}
\usepackage{amsmath}
\usepackage{amssymb}
\usepackage{amsfonts}
\usepackage{lipsum}
\usepackage{epstopdf}
\usepackage{algorithmic}
\usepackage{hyperref}
\usepackage{color}
\usepackage{mathptmx} 
\newtheorem{lem}{Lemma}
\newtheorem{thm}{Theorem}

\DeclareMathAlphabet{\bit}{OML}{cmm}{b}{it}

\def\<{\leqslant}           
\def\>{\geqslant}           
\def\div{\mathrm{div}}         

\def\d{\partial}

\def\cH{\mathcal{H}}   
\def\mZ{\mathbb{Z}}    
\def\mR{\mathbb{R}}    
\def\mC{\mathbb{C}}    

\def\Tr{\mathrm{Tr}}       
\def\rT{\mathrm{T}}        
\def\rF{\mathrm{F}}        

\def\ad{\mathrm{ad}}       

\def\bE{\mathbf{E}}    

\def\[[[{[\![\![}
\def\]]]{]\!]\!]}

\def\bra{{\langle}}
\def\ket{{\rangle}}

\def\Bra{\Big\langle}
\def\Ket{\Big\rangle}

\def\re{\mathrm{e}}        
\def\rd{\mathrm{d}}        

\def\cL{\mathcal{L}}

\def\bD{\mathbf{D}}

\def\bR{\mathbf{R}}

\def\br{\mathbf{r}}
\def\x{\times}
\def\ox{\otimes}

\def\Argmin{\mathop{\mathrm{Argmin}}}

\def\fS{\mathfrak{S}}

\def\mP{\mathbb{P}}

\def\sH{\mathsf{H}}

\def\bH{\mathbf{H}}

\def\cG{\mathcal{G}}

\def\cA{\mathcal{A}}
\def\cB{\mathcal{B}}

\def\cov{\mathbf{cov}}

\def\cS{\mathcal{S}}
\def\cT{\mathcal{T}}

\def\mS{\mathbb{S}}
\def\mT{\mathbb{T}}
\def\mZ{\mathbb{Z}}

\def\diag{\mathop{\mathrm{diag}}}    

\begin{document}
\title{\vspace{-12mm}\bf\Large Renyi's Relative Entropy Lyapunov Functional for Convergence to Invariant Measure in Dissipative Stochastic Hamiltonian Systems}


\author{Igor G. Vladimirov\footnote{School of Engineering, Australian National University, ACT 2601, Canberra, Australia; 
e-mail: {\small\tt igor.g.vladimirov@gmail.com.}}}
\date{}

\maketitle
\renewcommand{\abstractname}{}
\begin{abstract}
\vspace{-12mm}
\noindent\textbf{Abstract:}  
This paper is concerned with multivariable stochastic Hamiltonian systems governed by an ordinary differential equation for the position and an Ito stochastic differential equation for the momentum. The momentum dynamics involve both conservative and nonconservative forcing including Langevin viscous damping and the Wiener process as a random force. The setting allows for rotational degrees of freedom,  a non-quadratic potential energy function and position-dependent mass, diffusion and damping matrices. In addition to modelling nonlinear physical dynamics subject to random forcing, the system is motivated by a stochastic optimisation viewpoint, where the search for global minima  of the potential energy (for example, using the gradient descent) is carried out through convergence to a Boltzmann equilibrium measure  which favours smaller values of this function. Under certain regularity conditions, we show that the $\chi^2$-divergence or, equivalently, Renyi's second-order relative entropy with respect to the invariant measure has a nonpositive time derivative and thus provides a Lyapunov functional candidate for this convergence. However, this derivative vanishes at some moments of time, which is related to the behaviour of the probability distribution of the momentum  conditioned on the position and has a link to a family of auxiliary quantum harmonic oscillator Hamiltonians parameterised by the position space of the underlying classical system. This position-momentum conditioning is used in order to prove that such pre-equilibrium break times in the entropy dissipation are isolated,  thus making the $\chi^2$-divergence a strictly decreasing Lyapunov functional for the Fokker-Planck-Kolmogorov equation which governs the joint position-momentum distribution in such systems. 

\vspace{3mm}
\noindent\textbf{MSC codes:}  
70H14, 34F05, 35Q84, 82B31, 82C31, 37H30, 82M60, 37J25, 94A17, 82C10, 81P16. 

\vspace{3mm}
\noindent\textbf{Keywords:} 
global optimisation; 
dissipative stochastic Hamiltonian system; 
Fokker-Planck-Kolmogorov equation; 
equilibrium measure; 
$\chi^2$-divergence;
Renyi's relative entropy; 
Lyapunov functional;   
quantum harmonic oscillator. 
\end{abstract}

\section{Introduction}
\label{sec:intro}

A common feature of classical mechanical systems and their electrical and electromechanical counterparts, which are modelled using Lagrangian and Hamiltonian dynamics, is that their  state involves generalised positions and velocities or momenta. This partitioning into two conjugate sets of qualitatively different dynamic variables is reflected in the symplectic structure of Hamiltonian equations of motion  for such systems \cite{A_1989}. It also manifests itself in the energy balance relations, which  describe the exchange between the potential and kinetic components of energy, as well as the energy  dissipation  through  mechanical friction or electrical resistance.  These energy transfer and dissipation relations, which take into account conservative and nonconservative forces across subsystems,    play an important role in control by interconnection for port-Hamiltonian \cite{VJ_2014} and negative imaginary \cite{LP_2008,P_2016} systems. The energy transfer and dissipation provide a mechanism which secures stability in such systems and is captured in Lyapunov and more general storage functions.  The dissipation relations, leading to nonpositive  time derivatives of these functions,   allow the system behaviour to be localised around  its equilibrium.  If the time derivative vanishes at a nonequilibrium state,  but the system does not ``dwell'' there forever, asymptotic  stability can be established by using the  Barbashin-Krasovskii-LaSalle (BKL) invariance principle \cite{BK_1952,L_1960}.  For example, the Hamiltonian, which  quantifies the total energy of the system, provides a natural Lyapunov function   candidate. Assuming the Langevin damping force \cite{Z_2001} proportional to the velocity, the Hamiltonian has zero time derivative (with a ``break'' in the energy dissipation) whenever the velocity vanishes. However, the conservative force from a nonzero potential energy gradient keeps pushing the system, the nonvanishing acceleration makes  the velocity become nonzero (which is accompanied by the potential energy flowing into the kinetic energy), and the total energy dissipation continues. 

This potential-kinetic or position-momentum  interplay is adopted from classical mechanics in the heavy-ball optimisation method \cite{P_1964}. The latter  (see also \cite{UPS_2022,WPUS_2024} for  recent developments on this topic) adds the inertia effect to the gradient descent algorithm, so that its behaviour acquires features of physical system dynamics. The addition of a discrete-time counterpart of the momentum variables allows this iterative  method to get out of traps created by local minima of the potential energy and thus better search the position space for global minima. Another resource for improving the performance in  the context of global optimisation is provided by artificially introduced randomness in the form of a random initial condition or a random noise added to the gradient descent or Hamiltonian dynamics with dissipation. 
In the resulting discrete or continuous time stochastic systems, the role of equilibria is played by invariant probability measures instead of critical points of the potential energy or Hamiltonian functions on the position or augmented position-momentum phase spaces. Despite the more complicated nature of such equilibria, the system parameters can be tuned so as to make the invariant measures favour lower values of the functions subject to minimisation. This stochastic optimisation paradigm is also employed in simulated annealing \cite{KGV_1983} equipped with an absolute temperature parameter which is gradually decreased to zero according to a particular cooling schedule,  thus imitating a heat treatment procedure.   

In  the global optimisation and physical modelling context, of relevance  is a class of multivariable continuous time dissipative stochastic  Hamiltonian (DSH) systems, which involve both conservative  and deterministic damping and random forces. The position of such a system  is governed by an ordinary differential equation (ODE), which relates the velocity  to the momentum through  a generalised mass matrix, while the momentum dynamics have the form  of an Ito  stochastic differential equation (SDE) driven by a standard Wiener process \cite{KS_1991}.  The drift of the momentum SDE consists of a conservative force (coming from the Hamiltonian and involving the gradient of the potential energy) and a Langevin viscous  damping force specified by a damping matrix. The mass,  damping and momentum diffusion matrices can be position-dependent, which reflects the influence of the current spatial configuration of the system (especially in the presence of rotational degrees of freedom, as in the case of pendulum-like mechanical systems \cite{LPS_1996,OPU_2012}) on the tensor of inertia, energy dissipation  rate and exposure to a random environment.  Together with a nonquadratic potential energy function, in general, this results in a nonlinear DSH system (which extends linear DSH systems \cite{VP_2018_ANZCC,VP_2020_ANZCC} with quadratic potentials and constant mass, damping and diffusion matrices).  Such systems, discussed in \cite{Soize_1994},  are applicable to modelling flexible mechanical structures arising, for example,  in molecular dynamics \cite{H_1986,T_2010} and aeroelasticity  \cite{I_2019}  problems (involving the interaction with an external solvent or fluid flow),  as well as  electrical circuits subject to thermal noise \cite{R_1974}.

Under a multivariate version of the Einstein relation \cite{K_1966} between the damping and diffusion matrices,  the DSH system has an invariant measure in the form of a Maxwell-Boltzmann distribution from equilibrium statistical mechanics \cite{ME_1981}. The invariant probability density function (PDF) is specified by a temperature parameter and the Hamiltonian in such a way that it favours points in the phase space with lower energy values, especially as the temperature goes to zero. It is a steady-state solution of the Fokker-Planck-Kolmogorov equation (FPKE) \cite{BKRS_2015,R_1996,S_2008}, which, in application to the DSH system, governs the time evolution of the joint PDF of the position and momentum variables constituting a diffusion process in the phase space. The relevance of the dynamic structure of the DSH system and the Maxwell-Boltzmann invariant PDF in the optimisation and modelling context makes the convergence to the invariant measure in such a system  an important issue.  A Lyapunov functional candidate for this convergence in Markov processes is provided by the Kullback-Leibler   relative entropy \cite{CT_2006} (with respect to the invariant measure) due to its property  of being a nonincreasing function of time.   However, in the DSH setting, the time derivative of the relative entropy can vanish at some moments of time before the equilibrium measure is reached. This  resembles the above mentioned time evolution of the Hamiltonian in the deterministic dissipative case, except that  the breaks in the Kullback-Leibler relative entropy dissipation happen when the momentum distribution, conditioned on the position, coincides with its invariant conditionally Gaussian counterpart (instead of those instants when the velocity vanishes).  This position-momentum conditioning was shown in \cite{V_2023_AHP} to play a central role not only in the entropy dissipation breaks but also for establishing the strict monotonicity of the Kullback-Leibler relative entropy in time despite those breaks. This was achieved through representing the FPKE in terms of two coupled partial differential equations (PDEs) for the marginal position PDF and the conditional momentum PDF (which factorise the joint position-momentum PDF) and decomposing the entropy into the corresponding position and momentum parts, followed by investigating their dynamics (including the entropy exchange between them) under those PDEs. 

An extension of these ideas to different criteria for convergence to equilibrium  is the main purpose of the present paper. More precisely, for a multivariable DSH system satisfying the damping-diffusion relation, we consider the $\chi^2$-divergence \cite{PW_2025} of the position-momentum distribution with  respect to the Maxwell-Boltzmann invariant measure.  This deviation functional coincides, up to an additive constant,  with the exponential of Renyi's second-order relative entropy \cite{R_1961} and admits a decomposition into   a position part (which depends on the marginal position PDF) and a momentum part (which involves the conditional momentum PDF).  Using the position-momentum conditioning, we show that, under certain regularity conditions,   the $\chi^2$-divergence, or equivalently, Renyi's relative entropy, has a nonpositive time derivative and thus provides another Lyapunov functional candidate for the convergence to the invariant measure. 
However, this derivative vanishes at the same break times, as for the above mentioned Kullback-Leibler relative entropy, that is, whenever the conditional momentum PDF coincides with its invariant counterpart everywhere in the phase space. Despite this common feature, Renyi's relative entropy dissipation obeys different equations, which,  in contrast to the Kullback-Leibler relative entropy case of \cite{V_2023_AHP}, reveal a connection with a family of auxiliary quantum harmonic oscillator \cite{S_1994} Hamiltonians parameterised by the position space of the classical DSH system being considered. Furthermore, the nondegeneracy of the ground states for  these quantum mechanical Hamiltonians plays an important role in the structure of the $\chi^2$-divergence dissipation breaks. This connection is similar to yet different from that with the quantum anharmonic oscillator arising in the context of noisy gradient descent control \cite{KMPV_2025}. 

We then compute the second and third-order time derivatives of the $\chi^2$-divergence and its position and momentum components at the break times. These higher-order dissipation relations show that any pre-equilibrium break in the $\chi^2$-divergence dissipation is followed by a decrease in its position component over a sufficiently short interval of time, which is accompanied by an increase in the momentum component in such a way that they behave asymptotically as quadratic parabolas, and the $\chi^2$-divergence behaves like a cubic parabola with a stationary point of inflection at the break time. As a result, the break times are isolated and thus form a countable set, which makes Renyi's relative entropy a strictly decreasing function of pre-equilibrium time. The leading coefficients of the above mentioned quadratic parabolas are equal in absolute value and have opposite signs, as if Renyi's relative entropy flowed from the momentum component of the $\chi^2$-divergence to the position component before the break time and in the opposite direction after it.  This manifestation of the BKL principle in the DSH setting  through the ``waterbed''-like  exchange between the position and momentum components of the $\chi^2$-divergence leads to its strict monotonicity in time and  establishes this deviation functional as a  fully-fledged Lyapunov functional for the FPKE governing the joint position-momentum distribution.

The paper is organised as follows.
Section~\ref{sec:stopt} mentions the potential  energy minimisation on the position space as a motivation for the gradient descent with a random initial condition in Section~\ref{sec:rand} and its noisy version in Section~\ref{sec:noisy} before proceeding to stochastic Hamiltonian position-momentum dynamics in Section~\ref{sec:SHD} with nonconservative deterministic  and random forcing. Section~\ref{sec:varlan}  provides a variational property of Langevin viscous damping in terms of a pointwise optimisation problem of  mean energy decay speedup regularised  by a quadratic penalty on the nonconservative force. Section~\ref{sec:posmom} uses the conditioning of the momentum on the position and the resulting factorisation of the joint position-momentum PDF in order to obtain a PDE  for the evolution of the conditional momentum mean.  Section~\ref{sec:invmeas} provides a damping-diffusion relation which makes the Maxwell-Boltzmann  equilibrium distribution an invariant measure for the system. Section~\ref{sec:chi2} describes the position-momentum decomposition of the $\chi^2$-divergence with respect to the invariant measure and obtains a dissipation relation for it,  which involves a connection with a family of quantum harmonic oscillator Hamiltonians  and proves that its time derivative is nonpositive. The breaks in this dissipation, when the time derivative of the $\chi^2$-divergence vanishes, are investigated in   
Section~\ref{sec:break} along with the local behaviour of its position and momentum components.  These results are used in Section~\ref{sec:strict} for showing that the break times are isolated and the $\chi^2$-divergence is in fact strictly monotonic in time, with the underlying dissipation mechanism being a manifestation of the BKL principle through Renyi's relative entropy exchange between the position and momentum components. Section~\ref{sec:conc} makes concluding remarks. In addition to the
proofs of lemmas and theorems in the main body of the paper, Appendices \ref{sec:Morse}--\ref{sec:ground} provide auxiliary material on common Morse theoretic features   of the Hamiltonian and potential energy functions, integration  by parts under expectation,  and quantum harmonic oscillator Hamiltonians.

\section{Motivation from Global Optimisation}
\label{sec:stopt}

As mentioned in Introduction, one of motivations for the subsequent discussion is provided by a global minimisation problem 
\begin{equation}
\label{Vmin}
  V(q)
  \longrightarrow 
  \inf,
  \qquad
  q \in S, 
\end{equation}
where $V$ is a given $\mR_+$-valued  potential energy function on a position space $S$.  
The latter is assumed to be either  $\mR^n$, the $n$-dimensional  torus $\mT^n$,  or, more generally, the product $\mR^{n-r}\x \mT^r$ of such sets, with $\mT:= \mR / (2\pi \mZ)$   being identified with 
the interval $[0,2\pi)$, whereby functions on $\mT$ can be viewed as $2\pi$-periodic functions on the real line $\mR$ 
(here, $\mZ$ is the set of integers). Accordingly, growth conditions for functions at infinity in $S$ will apply only with respect to  the Euclidean position coordinates and are replaced with periodic boundary conditions over the $\mT$-valued coordinates. 
Such periodicity arises when the argument $q:= (q_k)_{1\< k \< n}$ of the function $V$ involves angular coordinates associated with rotational degrees of freedom as, for example, in the case of pendulum-like mechanical systems \cite{LPS_1996,OPU_2012} or molecular dynamics models \cite{H_1986}, where the spatial configuration is specified by dihedral angles. In each of these cases, the set $S$ is equipped with the usual  $n$-dimensional Lebesgue measure. 

A numerical solution of the optimisation problem (\ref{Vmin}) is provided by deterministic algorithms, such as the 
discrete-time Newton and gradient descent iterations. A continuous-time  counterpart of the latter is given by 
\begin{equation}
\label{Vgrad}
  \dot{Q}_t = a(Q_t),
  \qquad
  a:= -\nabla V 
\end{equation}
(with $\dot{(\ )}:= \frac{\rd }{\rd t}$ the time derivative) and produces 
a function $Q:= (Q_t)_{t\> 0} \in C^1(\mR_+, S)$ of time $t\> 0$. Here, $V$  is assumed to be twice continuously differentiable, $V \in C^2(S, \mR_+)$, and its gradient vector field $\nabla V:= \d_q V = (\d_{q_k}V)_{1\< k \< n}: S \to \mR^n$ over the position variables is assumed to be Lipschitz in order to guarantee the existence and uniqueness of solutions for the ODE  (\ref{Vgrad}) without a finite-time blowup. The Lipschitz property of $\nabla V$ is guaranteed, for example, by a sufficient condition 
\begin{equation}
\label{V''sup}
    \mu:= 
    \sup_{q \in S} \br(V''(q)) < +\infty, 
\end{equation}
where $V'':= \d_q^2 V=(\d_{q_j}\d_{q_k} V)_{1\< j,k\< n}$ is the Hessian matrix of $V$, and $\br(\cdot)$ is the spectral radius of a matrix (which coincides with the matrix operator norm for real symmetric matrices).

The potential energy function $V$ provides a natural Lyapunov function candidate for the ODE (\ref{Vgrad}) due to the well-known dissipation relation
\begin{equation}
\label{Vdot}
    \frac{\rd }{\rd t}V(Q_t)
    = 
    a(Q_t)^\rT \nabla V(Q_t) 
    = - |\nabla V(Q_t)|^2 
\end{equation}
(with $|\cdot|$ the standard  Euclidean norm), whose right-hand side is always non-positive and vanishes only at critical points of $V$ (where $\nabla V =0$). Such points are equilibria of the dynamical system (\ref{Vgrad}),  and their local stability is determined completely by the Hessian matrix $V''$,  provided $V$ is a Morse function \cite{M_1963}. That is, in addition to the above assumption that $V\in C^2(S, \mR_+)$  (smoothness in the form of infinite  differentiability is not needed here), $\det V'' \ne 0$ whenever  $\nabla V=0$. These conditions guarantee that all the critical points of $V$ are isolated and thus form a countable set in $S$.

Despite the deterministic nature of the dynamics (\ref{Vgrad}), the function of time $Q$ can also be regarded as a degenerate diffusion process, whose infinitesimal generator acts on functions $\varphi \in C^1(S,\mR)$ as 
\begin{equation}
\label{cG}
  \cG(\varphi)
  := 
  a^\rT 
  \nabla \varphi
  =
  -\nabla V^\rT  \nabla \varphi
\end{equation}
and is applied  in (\ref{Vdot}) to $V$. 
The formal adjoint of the operator $\cG$  with respect to the inner product $\bra u, v\ket := \int_S u(q)v(q)\rd q$ in the Hilbert space $L^2(S,\mR)$ (or, more generally, for functions $u, v: S \to \mR$ with an integrable pointwise product $uv \in L^1(S,\mR)$)  is given by 
\begin{equation}
\label{cG+}
  \cG^\dagger(\varphi)
  := 
  -\div(\varphi a) = -\cG(\varphi) - \varphi \div a .  
\end{equation}
Here, $\div(\cdot)$ is the divergence operator, and hence, in view of (\ref{Vgrad}), 
\begin{equation}
\label{diva}
    \div a = -\Delta V, 
\end{equation}
where $\Delta(\cdot):= \sum_{k=1}^n \d_{q_k}^2(\cdot)$ is the Laplacian, so that $\Delta V = \Tr V''$.

\section{Random Initial Condition}
\label{sec:rand}

If the initial condition $Q_0$ is allowed to be varied in order for the gradient descent flow (\ref{Vgrad})  to better explore the position space $S$ in searching for global minima of $V$,   the process $Q$  can be made random by initialising  the deterministic ODE at an absolutely continuously distributed  random vector $Q_0$ on $S$ with a PDF $f_0: S\to \mR_+$.  Then,  at any time $t\> 0$, the random vector 
\begin{equation}
\label{QtPhi}
    Q_t = \Phi_t(Q_0)
\end{equation}
has a PDF $f_t: S \to \mR_+$ which  is related to its initial condition $f_0$ as
\begin{equation}
\label{ftgrad}
  f_t(q)
  =
  f_0(\Phi_t^{-1}(q))\re^{-L_t(\Phi_t^{-1}(q))}, 
  \qquad
  q \in S,  
\end{equation}
where
\begin{equation}
\label{Lt}
  L_t(q):= \ln \det \Phi_t'(q). 
\end{equation}
Here, $(\Phi_t)_{t\> 0}$ is the semi-group of invertible maps on $S$ satisfying $\Phi_t, \Phi_t^{-1} \in C^1(S, S)$ and generated by the vector field $a$ in (\ref{Vgrad}) as 
\begin{equation}
\label{Phidot}
    \d_t \Phi_t(q) = a(\Phi_t(q)),
    \qquad
    \Phi_0(q) = q.   
\end{equation}
Also, $\Phi_t':= \d_q\Phi_t: S\to \mR^{n\x n}$ is the Jacobian matrix of $\Phi_t$, so that $\Phi_0' = I_n$ is the identity matrix of order $n$. Accordingly, in view of (\ref{Phidot}), the function $L_t$ in (\ref{Lt})  satisfies the initial value problem
\begin{align}
\nonumber
    \d_t L_t 
    & =
    \Tr ((\Phi_t')^{-1}\d_t \Phi_t')
    =
    \Tr (a'\circ \Phi_t)\\
\label{Ldot}
    & =
    (\div a)\circ \Phi_t
    =
     -(\Delta V)\circ \Phi_t,
     \qquad
     L_0 = 0,  
\end{align}
where use is also made of (\ref{diva}). 
If the initial PDF $f_0$ is continuously differentiable, $f_0 \in C^1(S,\mR_+)$, then $f_t$ is continuously  differentiable in time $t\> 0$  and space 
and satisfies the first-order PDE 
\begin{equation}
\label{fdotgrad}
  \d_t f_t = \cG^\dagger(f_t) = -\div U_t, 
  \qquad
  U_t := f_t a, 
\end{equation}
which is a reduced version of the FPKE,  
where the operator $\cG^\dagger$ from (\ref{cG+}) acts over the spatial argument of $f_t$. Here, the time-varying vector field $U_t: S \to \mR^n$ is the probability flux associated with the right-hand side of (\ref{Vgrad}) and the PDF $f_t$, so that (\ref{fdotgrad}) is equivalent to the continuity equation $\d_t f_t + \div U_t = 0$.  
This allows the energy dissipation (\ref{Vdot}) to be looked at from a dual viewpoint which is concerned with ``shrinkage'' of the PDF $f_t$ in (\ref{ftgrad}) for the random vector $Q_t$ in (\ref{QtPhi}) over the course of time in terms of its differential  entropy \cite{CT_2006}
\begin{align}
\nonumber
    \bH(f_t) 
    & := 
    -\int_S 
    f_t(q)\ln f_t(q)\rd q
    =
    -\bra f_t, \ln f_t\ket  \\
\nonumber
    & = 
    - \bE \ln f_t(Q_t) 
    = 
    -\bE \ln (f_0(Q_0) \re^{-L_t(Q_0)})
    \\
\label{bH}
    & = 
    \bH(f_0) + \bE L_t(Q_0) 
\end{align}
(with the usual convention that $0\ln 0 = 0$ by continuity),  
where $\bE(\cdot)$ is the expectation. The quantity (\ref{bH}) provides an upper bound for Renyi's entropy \cite{R_1961}  of order $\alpha>1$ given by 
\begin{equation}
\label{bHalf}
    \bH_\alpha(f_t) 
    := 
    \frac{1}{1-\alpha}
    \ln \bE (f_t(Q_t)^{\alpha-1})
    \< \bH(f_t)  
\end{equation}
and is its limiting case,  
so that $\bH(f) = \lim_{\alpha\to 1} \bH_\alpha(f)$ for any PDF $f: S \to \mR_+$ satisfying suitable regularity conditions.   Note that 
\begin{equation}
\label{falf}
    \bE (f_t(Q_t)^{\alpha-1}) = 
    \int_S 
    f_t(q)^\alpha 
    \rd q
    =
    \bra f_t^\alpha, 1\ket
\end{equation}
in (\ref{bHalf}) coincides with the $\alpha$th power of the $L^\alpha$-norm of the PDF $f_t$. 

\begin{lem}
\label{lem:shrink}
Suppose the potential energy  function $V \in C^2(S,\mR_+)$ 
satisfies (\ref{V''sup}),  
and the random initial condition $Q_0$ of the gradient descent flow (\ref{Vgrad})  has a PDF $f_0 \in L^\alpha(S,\mR_+)$, with $\alpha>1$. Then  
Renyi's entropy (\ref{bHalf}) for the PDF $f_t$ of $Q_t$ evolves in time $t\> 0$ as 
\begin{equation}
\label{bHalfdot}
  \frac{\rd}{\rd t}
  \bH_\alpha(f_t)
  =
  -\re^{(\alpha-1)\bH_\alpha(f_t)}
    \bE 
    (f_t(Q_t)^{\alpha-1} 
    \Delta V(Q_t)).  
\end{equation}
\end{lem}

\begin{proof}
Differentiation of (\ref{bHalf}) with respect to time leads to 
\begin{equation}
\label{bHdot1}
  \frac{\rd}{\rd t}
  \bH_\alpha(f_t)    
  = 
  \frac{1}{1-\alpha} 
  \re^{(\alpha-1)\bH_\alpha(f_t)}
  \frac{\rd}{\rd t}
  \bra
    f_t^\alpha,
    1
  \ket.  
\end{equation}
Here, the rightmost time derivative can be computed by representing (\ref{falf}) as
\begin{align}
\nonumber
    \bra f_t^\alpha, 1\ket 
    & = 
    \bE (f_0(Q_0)^{\alpha-1}\re^{(1-\alpha) L_t(Q_0)})\\
\label{falf1}
     & = 
     \bra 
        f_0^{\alpha}, 
        \re^{(1-\alpha) L_t}
     \ket, 
\end{align}
which employs (\ref{QtPhi}), (\ref{ftgrad}) similarly to (\ref{bH}).  Since $|\Delta V| = |\Tr V''| \< n \mu$ due to (\ref{V''sup}), then 
\begin{equation}
\label{expup}
    \re^{(1-\alpha) L_t(q)} \< \re^{(\alpha-1)n\mu t},
    \qquad
    t\> 0,\
    q \in S
\end{equation}
in view of (\ref{Ldot}). In particular, (\ref{falf1}), (\ref{expup}) imply that the $L^\alpha$-norm of the PDF  $f_t$ remains finite over the course of time, inheriting this property from $f_0 \in L^\alpha(S,\mR_+)$. By a dominated convergence argument, a combination of (\ref{falf1}) with (\ref{Ldot}) yields
\begin{align}
\nonumber
    \frac{1}{\alpha-1}
  \frac{\rd}{\rd t}
  \bra
    f_t^\alpha,
    1
  \ket
  & = 
    -
     \bra 
        f_0^{\alpha}, 
        \re^{(1-\alpha) L_t}
        \d_t L_t
     \ket   \\
\nonumber
  & = 
     \bra 
        f_0^{\alpha}, 
        \re^{(1-\alpha) L_t}
        (\Delta V) \circ \Phi_t
     \ket   \\
\nonumber
  & = 
     \bra 
        f_0^{\alpha}\re^{(1-\alpha) L_t}, 
        (\Delta V) \circ \Phi_t
     \ket   \\
\nonumber
  & = 
     \bra 
        (f_0\re^{-L_t})^\alpha\circ \Phi_t^{-1} , 
        \Delta V
     \ket        \\
\nonumber
  & = 
     \bra 
        f_t^\alpha, 
        \Delta V
     \ket  \\
\label{braket}
    & =
    \bE 
    (f_t(Q_t)^{\alpha-1} 
    \Delta V(Q_t)), 
\end{align}
where (\ref{QtPhi}), (\ref{ftgrad}) are used again. Substitution of (\ref{braket}) into (\ref{bHdot1}) yields (\ref{bHalfdot}). 
\end{proof} 

Note that in the case of continuously differentiable PDFs, the relation (\ref{braket}) can also be obtained by using the PDE (\ref{fdotgrad}) along with the generator $\cG$ and its adjoint $\cG^\dagger$ from (\ref{cG}), (\ref{cG+}) 
and the integration by parts based on the identity $\alpha f^{\alpha-1} \div(f\nabla V) = (\alpha-1) f^\alpha \Delta V + \div(f^\alpha \nabla V)$ for any  function $f\in C^1(S,\mR_+)$.

In the limiting case, as $\alpha\to 1$, the entropy dissipation relation (\ref{bHalfdot}) pertains to the  differential entropy (\ref{bH}) and acquires the form 
\begin{equation}
\label{bHdot}
  \frac{\rd}{\rd t}
  \bH(f_t)
  =
  -
    \bE 
    \Delta V(Q_t) . 
\end{equation}
The relations (\ref{bHalfdot}), (\ref{bHdot}) show that in the Euclidean case of $S = \mR^n$, if the potential energy function $V$ is subharmonic \cite{D_1984}, that is, $\Delta V \> 0$ everywhere in $S$ (in particular, if $V$ is convex), then both Renyi's entropy $\bH_\alpha(f_t)$ and the differential entropy $\bH(f_t)$ are nonincreasing functions of time $t\> 0$. Moreover, if the subharmonicity admits a lower bound in the sense that 
\begin{equation}
\label{sub}
    \inf_{q \in S} \Delta V(q) \> \sigma >0
\end{equation}
then, by applying the Gronwall-Bellman lemma to (\ref{braket}), it follows from     (\ref{bHalf}) that 
$$
    \bH_\alpha(f_t) 
    \< 
    \frac{1}{1-\alpha}
    \ln 
    (
    \bra f_0^\alpha, 1\ket \re^{\sigma (\alpha-1)t}
    )
    =    
    \bH_\alpha(f_0) - \sigma t.  
$$
Therefore, a similar inequality holds for the differential entropy (\ref{bH}) as well: 
$$
    \bH(f_t) 
    \< 
    \bH(f_0) - \sigma t ,
    \qquad
    t \> 0.   
$$
For example, if $V$ is strongly convex,   (\ref{sub}) is valid with  
$\sigma:= n\inf_{q \in S} \lambda_{\min}(V''(q))>0$, where $\lambda_{\min}(\cdot)$ is the smallest eigenvalue.   
However, the entropy decay to $-\infty$ does not give any information about the points of concentration of the probability distribution of $Q_t$ in $S$, as $t \to +\infty$, because Renyi's entropy (\ref{bHalf}) and its limiting case (the differential entropy (\ref{bH})) use the Lebesgue measure as a reference measure which is unrelated to the function $V$. Most importantly, multi-extremum  functions $V$, for which optimisation problems are particularly challenging, are not convex or subharmonic.

\section{Noisy Gradient Descent}
\label{sec:noisy}

Stochastic counterparts of deterministic optimisation algorithms include the noisy gradient descent and simulated annealing \cite{KGV_1983}. These  stochastic optimisation approaches replace the search for minima with a stochastic system evolving towards an invariant measure. Such a measure is usually organised as Boltzmann's  thermodynamic equilibrium distribution \cite{LL_1969,ME_1981}  in order to favour lower values of $V$ as specified by an absolute temperature parameter. In turn, the latter can be gradually decreased to zero according to a cooling schedule, which imitates the underlying physical process.

For example, one of continuous-time stochastic optimisation procedures for solving the problem (\ref{Vmin})  is provided by an Ito SDE 
\begin{equation}
\label{dQ}
  \rd Q_t 
  = 
  -\nabla V(Q_t)
  \rd t + \sqrt{2\cT} \rd W_t, 
  \qquad
  t \> 0, 
\end{equation}
as a noisy version of (\ref{Vgrad}), with a temperature parameter $\cT>0$. It produces an $S$-valued  diffusion process $Q$,  which is driven by a standard Wiener process $W:= (W_t)_{t \> 0}$ in $\mR^n$ independent  of the initial condition $Q_0$. Under the condition of finiteness of the statistical mechanical partition function \cite{ME_1981}
\begin{equation}
\label{Zbet}
    Z(\beta)
    := \int_S \re^{-\beta V(q)}
    \rd q
\end{equation}
for all sufficiently large values of the 
inverse temperature parameter
\begin{equation}
\label{betau}
    \beta := \frac{1}{\cT},   
\end{equation}
that is, for $\cT>0$ small enough,  
the process $Q$ in (\ref{dQ}) has an invariant Boltzmann PDF 
\begin{equation}
\label{f*grad}
    f_*(q)
    =
    \frac{1}{Z(\beta)}
    \re^{-\beta V(q)},
    \qquad
    q \in S  
\end{equation}
(see, for example, \cite{KMPV_2025} and references therein). This PDF is a normalised steady-state solution of the  FPKE in the sense that  
$
    \cL^\dagger(f_*) = 0 
$, 
with 
\begin{equation}
\label{cLgrad+}
    \cL^\dagger(\varphi) = \div(\varphi \nabla V) + \cT \Delta \varphi  
\end{equation}
the formal $L^2(S, \mR)$-adjoint of the infinitesimal generator $\cL$  of the process $Q$ acting as 
\begin{equation}
\label{cLgrad}
    \cL(\varphi) = \cT \Delta \varphi -\nabla V ^\rT \nabla \varphi  
\end{equation}
on functions $\varphi \in C^2(S,\mR)$.    
The solution of the ODE (\ref{Vgrad}), as a degenerate diffusion process, is  obtained by letting $\cT=0$ in (\ref{dQ}), in which case $\cL$, $\cL^\dagger$  in   (\ref{cLgrad}), (\ref{cLgrad+}) become the first-order differential operators $\cG$, $\cG^\dagger$ in (\ref{cG}), (\ref{cG+}), respectively.  The limiting gradient descent process $Q$ is a deterministic function of time and the initial condition $Q_0$, so that the latter provides the only source of randomness in that case.   However, the PDF $f_0$ of $Q_0$ as an $S$-valued random vector  can be such that it assigns a positive probability to the basins of attraction by local, rather than global, minima of $V$ for the gradient descent flow  (\ref{Vgrad}). The additional randomness,  which is brought about by the driving noise $W$ in the system dynamics (\ref{dQ}) with $\cT >0$,  promotes the ability of the diffusion process $Q$ to get out of such traps.  This can be seen from the relation 
\begin{align}
\nonumber
    \frac{\rd }{\rd t}
    \bE V(Q_t)
    & = 
    \bE \cL(V)(Q_t)\\
\label{EVdot}
    & =
    \bE (\cT \Delta V(Q_t) - |\nabla V(Q_t)|^2)
\end{align}
for the mean potential decay rate which employs (\ref{cLgrad}). Indeed, in sufficiently small neighbourhoods of strong local minima of $V$, its Hessian matrix satisfies $V''\succ 0$ and hence, $\Delta V >0$, so that the Ito correction  term $\cT \Delta V$ on such subsets of $S$ makes a positive contribution to the right-hand side of (\ref{EVdot}).

\section{Stochastic Hamiltonian Dynamics}
\label{sec:SHD}

Yet another resource for traversing  the neighbourhoods of local minima of the potential $V$  (similar to the inertia effect adopted from classical mechanics in the heavy-ball optimisation method \cite{P_1964,UPS_2022,WPUS_2024}) is provided by using an additional $\mR^n$-valued momentum process $P:= (P_t)_{t\> 0}$.   It gives rise to an augmented $S\x \mR^n$-valued diffusion process $X:= (X_t)_{t\> 0}$ given by 
\begin{equation}
\label{XQP}
  X_t
  := 
  \begin{bmatrix}
    Q_t\\
    P_t 
  \end{bmatrix}. 
\end{equation}
The latter is the state of a stochastic Hamiltonian system which models physical dynamics subject to a nonconservative deterministic force, specified by a function $u_t \in C^1(S \x \mR^n, \mR^n)$,   and a random force (from the standard Wiener process $W$ in $\mR^n$ mentioned before) governed by an ODE and an Ito SDE: 
\begin{align}
\label{Qdot}
    \dot{Q}_t 
    =& \d_p H(X_t) = M(Q_t)^{-1} P_t, \\
\label{dP}  
    \rd P_t
    = &  
    (-\d_q H(X_t) + u_t(X_t))\rd t
    + \sqrt{G(Q_t)} \rd W_t.  
\end{align}
The position process $Q:= (Q_t)_{t\> 0}$ has continuously differentiable trajectories, while $P$ is an Ito process   \cite{KS_1991} with a diffusion matrix $G:= (G_{jk})_{1\< j,k\< n}: S\to \mP_n$,  
where $\mP_n$ is the set of positive definite matrices in the space $\mS_n$ of real symmetric matrices of order $n$.  Here, $H: S \x \mR^n \to \mR_+$ is the system Hamiltonian which quantifies the total energy consisting of the potential energy 
$V$
and the kinetic energy $T:  S \x \mR^n \to \mR_+$,  so that  
\begin{equation}
\label{HVT}
    H(x) := V(q) + T(x),
    \qquad
    q \in S,\
    p \in \mR^n, 
\end{equation}
where $x \in S \x \mR^n$ is the augmented position-momentum vector  
\begin{equation}
\label{xqp}
    x
    :=
    (x_k)_{1\< k \< 2n}
    :=
    \begin{bmatrix}
      q\\
      p
    \end{bmatrix}. 
\end{equation}
The kinetic energy function is a position-dependent positive definite quadratic form 
\begin{equation}
 \label{T}
    T(x)
    :=
    \frac{1}{2} \|p\|_{M(q)^{-1}}^2
    =
    \frac{1}{2}
    \|\dot{q}\|_{M(q)}^2
\end{equation}
in the momentum variables  or the corresponding velocity vector 
\begin{equation}
\label{qdot}
  \dot{q}:= M(q)^{-1} p, 
\end{equation}
specified by a generalised mass matrix $M: S \to \mP_n$,  
where use is also made of  a weighted Euclidean norm  $\|v\|_N := \sqrt{v^\rT N v}$ of a real vector $v$ associated with an appropriately  dimensioned matrix $N\succ 0$. For what follows, the potential energy and the mass and diffusion matrices are assumed to be infinitely differentiable:
\begin{equation}
\label{VMGsmooth}
  V \in C^\infty(S,\mR_+),
  \qquad
  M, G \in C^\infty(S,\mP_n).   
\end{equation}
In the above mentioned case of rotational degrees of freedom,  where $\dot{q}_1, \ldots, \dot{q}_n$ are angular velocities, the matrix $M$ consists of the components of the tensor of inertia. In the general setting,  the gradients 
\begin{align}
\nonumber
    \d_q H
     & =
    \nabla V(q)
    -
    \frac{1}{2}
    (
        p ^\rT M_k(q) p
    )_{1\< k\< n}    \\
\label{dHdq}
    & = 
    \nabla V(q)
    -
    \frac{1}{2}
    (
        \dot{q}^\rT \d_{q_k}M(q) \dot{q}
    )_{1\< k\< n},\\
\label{dHdp}    
    \d_p H
    & = \dot{q}
\end{align}
of the Hamiltonian (\ref{HVT}) over the position and momentum variables (expressed in terms of (\ref{qdot})) form its full gradient 
\begin{equation}
\label{H'}
    H'
    : =
    \d_x H
    =
    \begin{bmatrix}
        \d_q H \\
        \d_p H
    \end{bmatrix} 
\end{equation}
with respect to all of the variables from (\ref{xqp}) in the phase space $S \x \mR^n$. The rightmost ``centrifugal'' term in (\ref{dHdq}) employs auxiliary functions $M_1, \ldots, M_n\in C^\infty(S, \mS_n)$,  which are  defined  by 
\begin{align}
\nonumber
    M_k(q)
    & := 
    -\d_{q_k} (M(q)^{-1}) \\
\label{Mk}
    & = 
    M(q)^{-1} 
    \d_{q_k}M(q) 
    M(q)^{-1}  
\end{align}
and combined with (\ref{qdot}),  so that the centrifugal term vanishes in the case of a constant mass matrix $M$. 
The equations (\ref{Qdot}), (\ref{dP}) can be assembled into one SDE for the state process (\ref{XQP}):
\begin{equation}
\label{dX}
    \rd X_t
     =
    a_t(X_t)
        \rd  t
        +
        b(Q_t)
        \rd W_t. 
\end{equation}
Its drift and diffusion terms are specified by maps $a_t: S\x \mR^n \to \mR^{2n}$ and $b: S \to \mR^{2n\x n}$ as
\begin{align}
\label{ax}
    a_t(x)
    & := 
        J H'(x)
        +
        \begin{bmatrix}
          0 \\
          1
        \end{bmatrix}
        \ox u_t(x),\\
\label{bq}
    b(q)
     & := 
        \begin{bmatrix}
          0 \\
          1
        \end{bmatrix} 
        \ox \sqrt{G(q)}
\end{align}
using (\ref{xqp}), (\ref{H'}) along with the symplectic structure matrix
\begin{equation}
\label{J}
    J:=
    \begin{bmatrix}
        0 & 1 \\
        -1 & 0
    \end{bmatrix}
    \ox I_n,  
\end{equation} 
where $\ox$ is the Kronecker product. In view of (\ref{VMGsmooth}), (\ref{bq}), the diffusion matrix $D \in C^\infty(S, \mS_{2n}^+)$ for the SDE (\ref{dX}) with values in the set $\mS_{2n}^+$ of real positive semi-definite symmetric matrices of order $2n$  takes the form 
\begin{equation}
\label{Dbb}
  D(q) = b(q)b(q)^\rT 
  = 
  \begin{bmatrix}
    0 & 0 \\
    0 & 1
  \end{bmatrix}
  \ox G(q), 
  \qquad
  q \in S,   
\end{equation} 
and is singular. Now, by (\ref{T}), the minimisation of the Hamiltonian (\ref{HVT}) is 
equivalent to that of the potential $V$ in (\ref{Vmin}):
\begin{align}
\label{HVmin}
    \inf_{x \in S\x \mR^n} H(x) & = \inf_{q \in S} V(q),\\
\label{HVargmin}
    \Argmin_{x \in S\x \mR^n} H(x) & = \Argmin_{q \in S} V(q) \x \{0\},  
\end{align}
see also Appendix~\ref{sec:Morse}.   In view of this equivalence, the deterministic function $u_t$, which plays the role of a control force in the drift of the SDE (\ref{dP}) (or (\ref{dX})),   can be chosen so as to accelerate the mean energy decay quantified by $\frac{\rd }{\rd t}\bE H(X_t)$.

\section{Variational Property of Langevin Damping}
\label{sec:varlan}

Consider a regularised version of the mean energy decay speedup using 
a weighted mean-square penalty on the nonconservative force $u_t$ in the stochastic Hamiltonian system (\ref{Qdot}), (\ref{dP}):
\begin{equation}
\label{EHdotmin}
    \frac{\rd }{\rd t}\bE H(X_t) + \frac{1}{2} \bE (\|u_t(X_t)\|_{F(Q_t)^{-1}}^2)
    \longrightarrow
    \inf, 
\end{equation}
where 
\begin{equation}
\label{Fsmooth}
    F\in C^\infty(S, \mP_n)
\end{equation}
is a given function of the position variables which specifies the force penalty, and the $\frac{1}{2}$-factor is introduced for convenience. The solution for the pointwise optimal control problem (\ref{EHdotmin}) is provided by Lemma~\ref{lem:uopt} below. Its formulation employs  
the infinitesimal generator $\cL_t$ of the diffusion process $X$ in (\ref{dX}),  which acts on a function $\varphi \in C^2(S\x \mR^n, \mR)$ as
\begin{align}
\nonumber
  \cL_t(\varphi)
\nonumber   
    &    =   
    a_t^\rT \varphi' + \frac{1}{2} \bra D, \varphi''\ket_\rF \\
\nonumber
    & = 
    \Big(
    -H'^\rT J
    + \begin{bmatrix}
      0 & u_t^\rT
    \end{bmatrix}
    \Big)
    \varphi'
    +
    \frac{1}{2}
    \bra
        G,
        \d_p^2 \varphi
    \ket_\rF \\    
\label{cL}
    & =
    \{\varphi,H\}
    + u_t^\rT\d_p\varphi
    +
    \frac{1}{2}
    \bra
        G,
        \d_p^2 \varphi
    \ket_\rF    
\end{align}
and inherits dependence on time from the force $u_t$ in (\ref{ax}).  
Here, the Frobenius inner product $\bra K, N\ket_\rF := \Tr (K^\rT N)$ of real matrices \cite{HJ_2007} is used along with (\ref{Dbb}) and  the Hessian matrix
\begin{equation*}
\label{Hess}
    \varphi''
    :=
    \d_x^2\varphi
    =
    \begin{bmatrix}
        \d_q^2 \varphi & \d_p\d_q \varphi  \\
        \d_q\d_p \varphi & \d_p^2 \varphi
    \end{bmatrix}
\end{equation*}
with respect to the position and momentum variables  from (\ref{xqp}). In (\ref{cL}), use is also made    of the 
Poisson bracket  \cite{A_1989}
\begin{align}
\nonumber
    \{\varphi, \psi\}
     & :=
    \varphi'^{\rT} J \psi'\\
\label{Poiss}
    & =
    \d_q \varphi^{\rT}\d_p \psi - \d_p \varphi^{\rT}\d_q \psi
    =
    -\{\psi,\varphi\}
\end{align}
for functions $\varphi, \psi \in C^1(S \x \mR^n, \mR)$,    where $J$ is the antisymmetric  matrix from (\ref{J}). 

\begin{lem}
\label{lem:uopt}
At any time $t \> 0$,   the optimal force in (\ref{EHdotmin}) is delivered by the time-invariant law
\begin{equation}
\label{uopt}
  u_t(x)
  = 
  - F(q) \dot{q},
  \qquad
  q \in S, 
\end{equation}
with $x \in S \x \mR^n$ and $\dot{q}\in \mR^n$ the position-momentum and velocity vectors from (\ref{xqp}), (\ref{qdot}). 
\end{lem}
\begin{proof}
As an infinitely differentiable function evaluated at the process (\ref{XQP}) governed by (\ref{dX})--(\ref{bq}) with the diffusion matrix (\ref{Dbb}), the Hamiltonian (\ref{HVT}) has the   stochastic differential
\begin{align}
\nonumber
    \rd H(X_t)
    & =
    H'(X_t)^{\rT}\rd X_t
    +
    \frac{1}{2} \bra D(Q_t), H''(X_t) \ket_\rF \rd t\\
\label{dH}
        & =
        \cL_t(H)(X_t)
        \rd t
        +
        \dot{Q}_t^\rT 
        \sqrt{G(Q_t)}\rd W_t,
\end{align}
obtained by applying the Ito lemma \cite{KS_1991} and using the generator $\cL_t$ from  (\ref{cL}) along with the relations (\ref{qdot}),  (\ref{dHdp}), (\ref{H'}).  
Since 
$$
    \{H,H\}
    =
    0
$$  
due to the antisymmetry of the Poisson bracket (\ref{Poiss}), the evaluation of the operator (\ref{cL}) at the Hamiltonian $H$ yields 
\begin{align}
\nonumber
    \cL_t(H)
    & =
    u_t^\rT\d_p H
    +
    \frac{1}{2}
    \bra
        G,
        \d_p^2 H
    \ket_\rF    \\
\label{cLH}
    & =
    u_t^\rT \dot{q}
    +
    \frac{1}{2}
    \bra
        G,
        M^{-1}
    \ket_\rF ,        
\end{align}
where use is also made of the Hessian matrix of $H$ over the momentum variables:
\begin{equation}
\label{Hpp}
    \d_p^2 H = \d_p^2 T = M^{-1}, 
\end{equation}
as follows from (\ref{HVT}), (\ref{T}), in accordance with (\ref{H''}).   
In view of (\ref{dH}), (\ref{cLH}),  the mean energy decay rate in (\ref{EHdotmin}) takes the form
\begin{align}
\nonumber
    \frac{\rd }{\rd t}
    \bE H(X_t)
    & = \bE \cL_t(H)(X_t)\\
\label{EHdot}
    & = 
    \bE 
    \Big(
    u_t(X_t)^\rT \dot{Q}_t
    +
    \frac{1}{2}
    \bra
        G(Q_t),
        M(Q_t)^{-1}
    \ket_\rF  
    \Big).        
\end{align}
Here, $u_t^\rT \dot{Q}_t$ is the power (work done per unit time) of the force $u_t$, while $\frac{1}{2}\bra G,M^{-1}\ket_\rF>0$ is the Ito correction term coming from the momentum diffusion matrix $G$ in (\ref{dP}) and the Hessian matrix (\ref{Hpp}) of the Hamiltonian. 
Substitution of (\ref{EHdot}) into the left-hand side of (\ref{EHdotmin})  allows the functional under minimisation to be represented as
\begin{align}
\nonumber
    \frac{\rd }{\rd t}
    \bE H(X_t)&+ \frac{1}{2} \bE (\|u_t(X_t)\|_{F(Q_t)^{-1}}^2)\\
\nonumber
     = &  
    \frac{1}{2}
    \bE
    \bra
        G(Q_t),
        M(Q_t)^{-1}
    \ket_\rF  \\
\nonumber
    & +
    \bE 
    \Big(    
        u_t(X_t)^\rT \dot{Q}_t
        + 
        \frac{1}{2} \|u_t(X_t)\|_{F(Q_t)^{-1}}^2
    \Big)\\
\nonumber
     = &  
    \frac{1}{2}
    \bE
    (
    \bra
        G(Q_t),
        M(Q_t)^{-1}
    \ket_\rF 
    -
    \|\dot{Q}_t\|_{F(Q_t)}^2
    ) \\
\nonumber
    & +
    \frac{1}{2}
    \bE 
    \big(    
        \|u_t(X_t)+ F(Q_t) \dot{Q}_t\|_{F(Q_t)^{-1}}^2
    \big)\\
\label{funmin}
    \> &  
    \frac{1}{2}
    \bE
    (
    \bra
        G(Q_t),
        M(Q_t)^{-1}
    \ket_\rF 
    -
    \|\dot{Q}_t\|_{F(Q_t)}^2
    )       
\end{align}
by using the completion of the square. 
The right-hand side of the inequality in (\ref{funmin}) provides the minimum value for (\ref{EHdotmin}) which is achieved at the force $u_t$ satisfying (\ref{uopt}) for almost all $x \in S\x \mR^n$ in the sense of the probability distribution of the system state $X_t$ at time $t\> 0$. 
\end{proof}

The relation (\ref{uopt}) describes the Langevin viscous damping force \cite{R_1996,Z_2001} which is proportional to the velocity, and the corresponding mean energy decay rate (\ref{EHdot}) acquires the form
\begin{equation}
\label{EHdot1}
    \frac{\rd }{\rd t}\bE H(X_t)
     = 
    \bE 
    \Big(
    \frac{1}{2}
    \bra
        G(Q_t),
        M(Q_t)^{-1}
    \ket_\rF  
    -
    \|\dot{Q}_t\|_{F(Q_t)}^2
    \Big).    
\end{equation}
 Therefore, Lemma~\ref{lem:uopt} gives a variational characterisation of such damping in terms of the pointwise optimisation problem (\ref{EHdotmin}) associated with the stochastic Hamiltonian setting.

In what follows, we will be concerned with the DSH system (\ref{Qdot}), (\ref{dP}) which has the Langevin damping force (\ref{uopt}) with a given damping matrix (\ref{Fsmooth}). Accordingly, the drift vector field (\ref{ax}) takes the form
\begin{equation}
\label{axF}
    a
     = 
        J H'
        -
        \begin{bmatrix}
          0 \\
          1
        \end{bmatrix}
        \ox (F\dot{q})
        =
        \left(
            J - 
            \begin{bmatrix}
                0 & 0\\
                0 & 1
            \end{bmatrix}
            \ox F
        \right)
        H'
\end{equation}
which corresponds to that used in the theory of port-Hamiltonian systems \cite{VJ_2014}, and   
the generator (\ref{cL}) is specified as
\begin{equation}
\label{cLF}
  \cL(\varphi)
    =
    \{\varphi,H\}
    - p^\rT M^{-1}F \d_p\varphi
    +
    \frac{1}{2}
    \bra
        G,
        \d_p^2 \varphi
    \ket_\rF .    
\end{equation}
The formal $L^2(S\x \mR^n, \mR)$-adjoint of the operator $\cL$ in (\ref{cLF}) is given by 
\begin{align}
\nonumber
    \cL^\dagger(\varphi)
    = & 
    -\div(\varphi a)  
    +
    \frac{1}{2}
    \div^2
    (
        \varphi D
    )\\    
\nonumber
    = &  
    \{H,\varphi\}
    +
    p^\rT M^{-1} F \d_p \varphi\\
\label{cL+}
    & +
    \bra
        F, M^{-1}
    \ket_\rF 
    \varphi
    +
    \frac{1}{2}
    \bra
        G, \d_p^2 \varphi
    \ket_\rF .  
\end{align}
Here, in accordance with (\ref{xqp}),   the divergence operator $\div(\cdot)$ acts on $\mR^{2n}$-valued vector fields  formed from $u,v \in C^1(S\x \mR^n, \mR^n)$ as
\begin{equation*}
\label{divuv}
    \div
    \begin{bmatrix}
      u\\
      v
    \end{bmatrix}
    =
    \div_q u + \div_p v, 
\end{equation*}
where $\div_q(\cdot)$, $\div_p(\cdot)$ are the divergence operators acting  on $\mR^n$-valued vector fields over the position and momentum variables, respectively. In application to matrix fields, these divergence operators act in a row-wise fashion and map them to vector fields. Accordingly, the composition $\div^2(\cdot) := \div(\div(\cdot))$ of the divergence operator with itself, which is used in (\ref{cL+}),  maps a symmetric matrix field  $w:= (w_{jk})_{1\< j,k\< 2} \in C^2(S\x \mR^n, \mS_{2n})$ (with blocks $w_{jk} = w_{kj}^\rT \in C^2(S\x \mR^n, \mR^{n \x n})$) to a scalar field 
\begin{align}
\nonumber
    \div^2 w 
    & = 
    \div
    \begin{bmatrix}
      \div_q w_{11} + \div_p w_{12}\\
      \div_q w_{21} + \div_p w_{22}
    \end{bmatrix}\\
\label{div2w}
    & =
    \div_q^2 w_{11} + 2\div_q \div_p w_{12} + \div_p^2 w_{22}.  
\end{align}
In particular, (\ref{div2w}) yields  
\begin{equation*}
\label{divDG}
    \div D 
    = 
    \begin{bmatrix}
         0 \\
         \div_p G
    \end{bmatrix}
    =
    0,
    \qquad
    \div^2 D = \div_p^2 G = 0 
\end{equation*}
by the sparsity of the matrix $D$ in (\ref{Dbb}) and  the $p$-indepen\-dence of $G$.

\section{Position-Momentum Conditioning}
\label{sec:posmom}

Note that, unlike its noisy gradient descent counterpart (\ref{cLgrad+}),  the operator $\cL^\dagger$ in (\ref{cL+}) is 
elliptic only with respect to the momentum variables, thus potentially affecting the smoothness properties for the solution of the FPKE (\ref{fdot}) which are usually due largely  to full ellipticity \cite{E_2008,S_2008}. 
We will therefore use an additional assumption that at any time $t\> 0$,  the state vector $X_t$ in (\ref{XQP}) is absolutely continuously distributed  with a PDF  $f_t:S\x \mR^n \to \mR_+$ (with respect to the $2n$-dimensional  Lebesgue measure on the phase space) which is infinitely differentiable in time and space:
\begin{equation}
\label{fsmooth}
    f_\bullet(\cdot)
    \in
    C^\infty(\mR_+\x S \x \mR^n, \mR_+). 
\end{equation}
These assumptions can be justified 
by using H\"{o}rmander's parabolic hypoellipticity  condition \cite[Theorem 1.1]{H_1967} (see also \cite[Theorems 7.4.3 and 7.4.18]{S_2008}),  whose verification for the system being considered is omitted for brevity. Then the PDF $f_t$ satisfies the FPKE 
\begin{align}
\nonumber
  \d_t f_t 
  = &    \cL^\dagger(f_t) \\
\nonumber  
    =&     
    \{H,f_t\}
    +
    p^\rT M^{-1} F \d_p f_t\\
\label{fdot}
    & +
    \bra
        F, M^{-1}
    \ket_\rF 
    f_t
    +
    \frac{1}{2}
    \bra
        G, \d_p^2 f_t
    \ket_\rF ,   
\end{align}
with the operator $\cL^\dagger$ from (\ref{cL+}) acting over the position and momentum variables of $f_t$. Furthermore, in order to make boundary effects vanish at infinity  in the integration by parts  over the phase space and justify related properties (such as differentiation with respect to  time under expectation), which are used in the subsequent results involving the PDF $f_t$, we assume that the latter belongs to the Schwartz space \cite{V_2002} of rapidly decreasing functions on $S\x \mR^n$ at every moment of time:
\begin{equation}
\label{fgood}
  f_t \in \cS(S\x \mR^n, \mR_+), 
  \qquad
  t\> 0.  
\end{equation}
This property is inherited by the marginal PDF $g_t: S \to \mR_+$  of the position $Q_t$ of the system, obtained by integrating the PDF $f_t$   over the momentum variables:
\begin{equation}
\label{gt}
    g_t(q) = \int_{\mR^n} f_t(q,p)\rd p,
    \qquad
    t \> 0,\ 
    q\in S,    
\end{equation}
so that $g_t \in \cS(S,\mR_+)$. 
Also, we assume that the position PDF is positive everywhere:
\begin{equation}
\label{gtpos}
  g_t(q) > 0,
    \qquad
    t \> 0,\ 
    q\in S.   
\end{equation}
Then the conditional PDF $h_t(\cdot \mid q): \mR^n\to \mR_+$ of the momentum $P_t$, given $Q_t=q$,  is provided by 
\begin{equation}
\label{ht}
  h_t(p\mid q)
  =
  \frac{f_t(q,p)}{g_t(q)} 
\end{equation}
and satisfies $h_t(\cdot \mid q) \in \cS(\mR^n,\mR_+)$ for any $t\> 0$, $q \in S$.  
This position-momentum conditioning allows the joint PDF $f_t$ of $Q_t$ and $P_t$ to be factorised as
\begin{equation}
\label{fght}
    f_t(x) = g_t(q)h_t(p\mid q),
    \qquad
    t \> 0,\ q\in S,\
    p \in \mR^n, 
\end{equation}
with $x$ given by (\ref{xqp}).  The smoothness properties are also inherited by 
the  conditional momentum  mean and covariance matrix
\begin{align}
\label{EPQ}
  \gamma_t(q)
  & :=
  \bE(P_t\mid Q_t=q)
  =
  \int_{\mR^n}
   h_t(p\mid q) p
  \rd p,\\
\nonumber
    \Sigma_t(q)
    & := 
    \cov(P_t\mid Q_t = q)\\
\nonumber
    & =
    \int_{\mR^n}
    (p-\gamma_t(q))
    (p-\gamma_t(q))^\rT 
    h_t(p\mid q)
    \rd p\\
\label{covPQ}
    & =
    \bE (P_tP_t^\rT \mid Q_t = q)- \gamma_t(q)\gamma_t(q)^\rT,    
\end{align}
so that   $\gamma_\bullet(\cdot) 
  \in 
  C^{\infty}(\mR_+\x S,\mR^n)$ and  
$
  \Sigma_\bullet(\cdot)
  \in
  C^{\infty}(\mR_+\x S,\mS_n^+)$.   
The conditional momentum mean $\gamma_t$ in (\ref{EPQ}) drives the evolution of the position PDF $g_t$ according to the PDE \cite[Lemma 3.2]{V_2023_AHP} 
\begin{equation}
\label{gdot}
  \d_t g_t(q) 
  = 
  - \div_q(g_t(q) M(q)^{-1}\gamma_t(q)), 
\end{equation}
which follows from the ODE (\ref{Qdot}) regardless of a particular structure of the momentum dynamics  (\ref{dP}).  The latter affect the PDE from \cite[Lemma 3.3]{V_2023_AHP} (given below for completeness of exposition and not used in what follows) for the conditional momentum PDF (\ref{ht}):
\begin{align*}
    \d_t h_t
    &=
    \{H,h_t\}
    +
    p^\rT M^{-1} F \d_p h_t
    +
    \frac{1}{2}
    \bra
        G, \d_p^2 h_t
    \ket_\rF\\
     & +
     (    \bra
        F, M^{-1}
    \ket_\rF
    -(p-\gamma_t)^\rT M^{-1}\d_q \ln g_t
    + \div_q (M^{-1}\gamma_t))
     h_t,  
\end{align*}
which is cross-coupled with (\ref{gdot}). 
Furthermore, the momentum dynamics  influence the evolution of $\gamma_t$ in (\ref{EPQ}), which is driven by the conditional covariance matrix (\ref{covPQ}) and the  conditional mean 
\begin{equation}
\label{Psit}
  \Psi_t(q)
  := 
  \bE (-\d_q H(X_t) + u_t(X_t)\mid Q_t = q)
\end{equation}
of the drift in the momentum SDE  (\ref{dP})  as shown in Theorem~\ref{th:gamdot} below. The $\mR^n$-valued function $\Psi_t$ 
involves contributions from the conservative and nonconservative forces and will be computed in Lemma~\ref{lem:Psi}. 

\begin{thm}
\label{th:gamdot}
Suppose the joint position-momentum PDF for the DSH system (\ref{Qdot}), (\ref{dP}) with the Langevin damping force (\ref{uopt})   
satisfies (\ref{fsmooth}), (\ref{fgood}), (\ref{gtpos}). Then  the conditional momentum mean (\ref{EPQ}) satisfies 
\begin{align}
\nonumber
    \d_t \gamma_t & + (\d_q \gamma_t) M^{-1} \gamma_t\\ 
\label{gamdot}
    & + \div_q(\Sigma_t M^{-1}) + \Sigma_t M^{-1} \d_q \ln g_t = \Psi_t, 
\end{align}
where $g_t$ is the position PDF (\ref{gt}),   $\d_q \gamma_t: S\to \mR^{n\x n}$ is the Jacobian matrix of $\gamma_t$, and the functions $\Sigma_t$, $\Psi_t$ are given by (\ref{covPQ}), (\ref{Psit}).   
\end{thm}

\begin{proof}
Let $\varphi \in C^\infty(S,\mR)$  be an arbitrary  infinitely differentiable test function of bounded support. Then,  by the properties of iterated conditional expectations \cite{S_1996} and (\ref{EPQ}), 
\begin{align}
\nonumber
    \bE (\varphi(Q_t) P_t ) 
    & = 
    \bE \bE (\varphi(Q_t) P_t\mid Q_t)
    = 
    \bE (\varphi(Q_t) \bE (P_t\mid Q_t))\\
\label{EphiQP}
    & =
    \bE(\varphi(Q_t) \gamma_t(Q_t)) 
\end{align}
at any time $t\> 0$.   
By a similar reasoning, differentiation of (\ref{EphiQP}) in time yields 
\begin{align}
\nonumber
    \frac{\rd }{\rd t}
    \bE (\varphi(Q_t) P_t )
    = & 
    \frac{\rd }{\rd t}\bE(\varphi(Q_t) \gamma_t(Q_t))\\
\nonumber
    = & \bE(\varphi(Q_t) (\d_t \gamma_t(Q_t) + \gamma_t{}'(Q_t)\dot{Q}_t)\\
\nonumber
     & + \gamma_t(Q_t) \dot{Q}_t^\rT \varphi'(Q_t)) \\
\nonumber
    = & \bE \bE(\varphi(Q_t) (\d_t \gamma_t(Q_t) + \gamma_t{}'(Q_t)\dot{Q}_t)\\
\nonumber
     & + \gamma_t(Q_t) \dot{Q}_t^\rT \varphi'(Q_t) \mid Q_t) \\     
\nonumber
    = & \bE (\varphi(Q_t) (\d_t \gamma_t(Q_t) + \gamma_t{}'(Q_t)\bE (\dot{Q}_t\mid Q_t))\\
\label{EphiQP1}
     & + \gamma_t(Q_t) \bE (\dot{Q}_t \mid Q_t)^\rT \varphi'(Q_t)), 
\end{align}
where, for the sake of brevity, we denote by  $\gamma_t{}'(q):= \d_q \gamma_t(q)$ the Jacobian matrix of $\gamma_t$ and by $\varphi'(q):= \nabla \varphi(q)$  the gradient of $\varphi$ with respect to  $q \in S$.  Here, 
\begin{equation}
\label{EQdotQ}
    \bE (\dot{Q}_t \mid Q_t) 
    = 
    \bE (M(Q_t)^{-1} P_t \mid Q_t)
    =
    M(Q_t)^{-1} \gamma_t(Q_t) 
\end{equation}
in view of (\ref{Qdot}), (\ref{qdot}), (\ref{dHdp}), (\ref{EPQ}).  
Hence, by substituting (\ref{EQdotQ}) into (\ref{EphiQP1}) and using the symmetry of the mass matrix $M$, it follows that 
\begin{align}
\nonumber
    \frac{\rd }{\rd t}
    \bE (\varphi(Q_t) P_t )
    = & \bE (\varphi(Q_t) (\d_t \gamma_t(Q_t) + \gamma_t{}'(Q_t)M(Q_t)^{-1} \gamma_t(Q_t))\\
\label{EphiQP2}
     & + \gamma_t(Q_t) \gamma_t(Q_t)^\rT M(Q_t)^{-1} \varphi'(Q_t)).  
\end{align}
While the time derivative in (\ref{EphiQP2}) is found by using the ODE (\ref{Qdot}), it can also be computed in an alternative way based on the momentum SDE (\ref{dP}). To this end, note that the $\mR^n$-valued function  $\varphi(q)p$ is linear in $p:= (p_k)_{1\< k\< n}\in \mR^n$, and hence, its stochastic differential has zero Ito correction term: $\frac{1}{2} \sum_{1\< j,k\< n}G_{jk}(q)\d_{p_j}\d_{p_k}(\varphi(q)p) = 0$ by the sparsity of the diffusion  matrix (\ref{Dbb}). Therefore, 
\begin{align}
\nonumber
    \frac{\rd }{\rd t}\bE (\varphi(Q_t) P_t )
    = & 
    \bE (P_t \dot{Q}_t^\rT \varphi'(Q_t)\\
\nonumber
    &  + \varphi(Q_t)(-\d_qH(X_t) + u_t(X_t)))\\
\nonumber
    = & 
    \bE \bE (P_t P_t^\rT M(Q_t)^{-1}\varphi'(Q_t)\\
\nonumber
    &  + \varphi(Q_t)(-\d_qH(X_t) + u_t(X_t))\mid Q_t)\\    
\nonumber
    = & 
    \bE ( \bE(P_t P_t^\rT\mid Q_t) M(Q_t)^{-1}\varphi'(Q_t)\\
\nonumber
    &  + \varphi(Q_t)\bE (-\d_qH(X_t) + u_t(X_t)\mid Q_t))\\        
\nonumber
    = & 
    \bE ( \bE(P_t P_t^\rT\mid Q_t) M(Q_t)^{-1}\varphi'(Q_t)\\
\label{EphiQP3}
    &  + \varphi(Q_t)\Psi_t(Q_t)). 
\end{align}
Since these two alternative ways (\ref{EphiQP2}), (\ref{EphiQP3}) of finding the same quantity must give identical results, their comparison leads to 
\begin{align}
\nonumber
    \bE &(\varphi(Q_t) (\d_t \gamma_t(Q_t)  + \gamma_t{}'(Q_t)M(Q_t)^{-1} \gamma_t(Q_t) - \Psi_t(Q_t))\\
\nonumber
     & = 
     \bE ((\bE(P_tP_t^\rT \mid Q_t)-\gamma_t(Q_t) \gamma_t(Q_t)^\rT) M(Q_t)^{-1}\varphi'(Q_t))\\
\label{EphiQP4}
    & = 
    \bE (\Sigma_t(Q_t)M(Q_t)^{-1}\varphi'(Q_t)), 
\end{align}
where use is also made of (\ref{covPQ}). Application of (\ref{exparts}) allows the right-hand side of (\ref{EphiQP4}) to be computed as 
\begin{align}
\nonumber
    \bE  (\Sigma_t(Q_t)&M(Q_t)^{-1}\varphi'(Q_t))\\
\nonumber
     = &   
     -\bE(\varphi(Q_t)(\div_q (\Sigma_t M^{-1})(Q_t)\\
\label{right}
      & + \Sigma_t(Q_t)M(Q_t)^{-1} \nabla \ln g_t(Q_t))) 
\end{align}
in view of $g_t>0$ as assumed in (\ref{gtpos}). 
Subtraction of the right-hand side of (\ref{right}) from the left-hand side of (\ref{EphiQP4}) yields
\begin{align}
\nonumber
    \int_S
     g_t(q)\varphi(q) &(\d_t \gamma_t(q)  + \gamma_t{}'(q) M(q)^{-1} \gamma_t(q)\\
    \nonumber 
    & - \Psi_t(q) 
    + \div_q(\Sigma_t(q) M(q)^{-1})\\
\label{gamdot1}
    & + \Sigma_t(q) M(q)^{-1} \nabla \ln g_t(q)) \rd q
    = 0.  
\end{align}
Since the function $\varphi$ is arbitrary,  and $g_t>0$, then (\ref{gamdot1}) implies (\ref{gamdot}). 
\end{proof}

Theorem~\ref{th:gamdot} is completed by the following calculation of the right-hand side of the PDE (\ref{gamdot}). 

\begin{lem}
\label{lem:Psi}
Under the conditions of Theorem~\ref{th:gamdot}, 
the function $\Psi_t: S \to \mR^n$ in (\ref{Psit}) is computed for any time $t\> 0$  as  
\begin{equation}
\label{Psi0}
  \Psi_t
  =   
  -\nabla V 
  +
  \frac{1}{2}
  \big(
  \bra
    M_k,
    \gamma_t \gamma_t^\rT + \Sigma_t 
  \ket_\rF
  \big)_{1\< k \< n}
  - FM^{-1}\gamma_t 
\end{equation}
in terms of the conditional momentum mean and covariance matrix from (\ref{EPQ}), (\ref{covPQ}) and the functions (\ref{Mk}). 
\end{lem}
\begin{proof}
In view of (\ref{qdot}), the conditional averaging of  (\ref{dHdq}) leads to 
\begin{align}
\nonumber
    \bE (\d_qH(X_t) &\mid Q_t = q)
    =  
    \nabla V(q)\\
\label{EdTQ}
    & -
  \frac{1}{2}
  \big(
  \bra
    M_k(q),
    \bE(P_tP_t^\rT \mid Q_t = q)
  \ket_\rF
  \big)_{1\< k \< n}     
\end{align}
for any $t\> 0$ and $q \in S$. The conditional second-moment matrix of the momentum in (\ref{EdTQ}) is related to 
(\ref{EPQ}), (\ref{covPQ}) by 
\begin{equation}
\label{EPPQ}
    \bE(P_tP_t^\rT \mid Q_t = q)
    =
    \gamma_t(q) \gamma_t(q)^\rT + \Sigma_t(q).  
\end{equation}
The conditional mean of the Langevin viscous damping force (\ref{uopt}) in (\ref{dP})  is found from 
\begin{equation}
\label{EFvQ}
  \bE (F(Q_t) M(Q_t)^{-1} P_t \mid Q_t = q)
  =
  F(q) M(q)^{-1} \gamma_t(q)
\end{equation}
using (\ref{EPQ}) again. By assembling (\ref{EdTQ})--(\ref{EFvQ}), it now follows that (\ref{Psit}) takes the form (\ref{Psi0}). 
\end{proof}

Therefore, in the case of Langevin damping, the conditional momentum mean function $\gamma_t$ evolves in time according to a PDE 
\begin{align}
\nonumber
    \d_t \gamma_t & + (\d_q \gamma_t) M^{-1} \gamma_t\\ 
\nonumber
    & + \div_q(\Sigma_t M^{-1}) + \Sigma_t M^{-1} \d_q \ln g_t\\
\label{gamdot2}
     = &  
      -\nabla V
  +
  \frac{1}{2}
  \big(
  \bra
    M_k,
    \gamma_t \gamma_t^\rT + \Sigma_t 
  \ket_\rF
  \big)_{1\< k \< n}
   - FM^{-1}\gamma_t  
\end{align}
obtained by combining (\ref{gamdot}) with (\ref{Psi0}). The particular structure of the momentum dynamics (\ref{dP}) enters (\ref{gamdot2}) only through its right-hand side. Moreover, the Langevin damping (\ref{uopt}) manifests itself in (\ref{gamdot2}) only through the term 
   $- FM^{-1}\gamma_t$, which is replaced with  $\bE (u_t(X_t)\mid Q_t=q)$  for a general nonconservative force, as can be seen from the proof of Lemma~\ref{lem:Psi}.

\section{Statistical Mechanical Invariant Measure}
\label{sec:invmeas}

For a given $\cT>0$ and the inverse temperature parameter $\beta$ from  (\ref{betau}) as before,   consider the Maxwell-Boltzmann PDF
\begin{equation}
\label{f*}
    f_*(x) 
    := 
    \frac{1}{\Pi(\beta)} \re^{-\beta H(x)},
    \qquad
    x \in S\x \mR^n,
\end{equation}
which, similarly to (\ref{f*grad}), also pertains to the statistical mechanical equilibrium \cite{ME_1981}, except that it is associated with the Hamiltonian (\ref{HVT}) rather than the potential energy $V$. In comparison with (\ref{Zbet}),  the partition function in (\ref{f*})  is appropriately modified: 
\begin{align*}
\nonumber
  \Pi(\beta)
  & :=
  \int_{S \x \mR^n}
  \re^{-\beta H(x)}
  \rd x\\
\nonumber
  & = 
  \int_S
  \re^{-\beta V(q)}
  \Big(
  \int_{\mR^n}
  \re^{-\frac{1}{2}\beta \|p\|_{M(q)^{-1}}^2}
  \rd p
  \Big)
  \rd q\\
\label{Pibet}
  & =
  (2\pi \cT )^{n/2}
  \int_S
  \re^{-\beta V(q)}
  \sqrt{\det M(q)}
  \rd q , 
\end{align*}
where the additive structure of $H$ is used  along with  the quadratic dependence of the kinetic energy (\ref{T}) on the momentum. As in the factorisation (\ref{fght}) of $f_t$  using $g_t$, $h_t$ from (\ref{gt}), (\ref{ht}), 
the equilibrium PDF (\ref{f*}) admits the factorisation 
\begin{equation}
\label{fgh*}
    f_*(x)
    =
    g_*(q) h_*(p\mid q),
    \qquad
    x \in S\x \mR^n,
\end{equation}
in terms of the equilibrium  position PDF $g_*$  computed as  
\begin{align}
\nonumber
  g_*(q)
  & =
  \int_{\mR^n}
  f_*(q, p)
  \rd p\\
\label{g*}
  & =
  \frac{(2\pi \cT )^{n/2}}{\Pi(\beta)}
  \re^{-\beta V(q)}
  \sqrt{\det M(q)},
  \qquad
  q \in S,
\end{align}
and the equilibrium conditional momentum PDF $h_*$ given the position:
\begin{equation}
\label{h*}
  h_*(p\mid q)
  =
  \frac{  (2\pi \cT )^{-n/2}}{\sqrt{\det M(q)}}
  \re^{-\frac{1}{2}\beta\|p\|_{M(q)^{-1}}^2},
  \quad
  p \in \mR^n,    
\end{equation}
which is a Gaussian PDF with zero mean and the conditional covariance matrix $\cT M(q)$. 

As discussed in \cite[Lemma 4.1]{V_2023_AHP} and can also be obtained from  \cite[Theorem 4 on p. 211]{Soize_1994},  the PDF (\ref{f*}) is an invariant PDF for the diffusion process $X$ in  (\ref{dX})--(\ref{bq}) with the Langevin damping force (\ref{uopt}), that is, a steady-state solution of the FPKE (\ref{fdot}), 
\begin{equation}
\label{cL+0}
    \cL^\dagger(f_*) = 0, 
\end{equation}
if and only if the damping and diffusion matrices in (\ref{axF}), (\ref{bq}) satisfy the following multivariate counterpart of the Einstein relation \cite[Eq. (3.14) on p. 260]{K_1966}:
\begin{equation}
\label{FG}
    F(q) = \frac{\beta}{2} G(q),
    \qquad
    q \in S.
\end{equation}
Under this condition, which is equivalent to $ G = 2\cT F$ in view of (\ref{betau}) and is assumed to be fulfilled in what follows,  the SDEs (\ref{dP}), (\ref{dX}) with the Langevin damping force (\ref{uopt}) acquire parametric dependence on $\cT$ (or $\beta$), and hence, so also do  the generator $\cL$ in  (\ref{cLF}) and its adjoint $\cL^\dagger$ in (\ref{cL+}).   
Accordingly, the FPKE (\ref{fdot}) for the system state PDF $f_t$ takes the form  
\begin{align}
\nonumber
  \d_t f_t
  = &  
  \cL^\dagger(f_t) \\
\nonumber
    = &    
    \{H,f_t\}
    +
    \frac{1}{2}
    \bra
        G, \d_p^2 f_t
    \ket_\rF    \\
\label{fdotbet} 
    & + 
    \frac{\beta}{2}
    (p^\rT M^{-1} G \d_p f_t
    +
        \bra
        G, M^{-1}
    \ket_\rF 
    f_t) 
\end{align}
and has the equilibrium PDF (\ref{f*}) as its steady-state solution in the sense of (\ref{cL+0}).

\section{Evolution of $\chi^2$-Divergence from Invariant PDF}
\label{sec:chi2}

Under the damping-diffusion relation (\ref{FG}), the deviation of the joint position-momentum PDF $f_t$ at time $t\> 0$  from the invariant PDF $f_*$ in   (\ref{f*}) can be quantified by the $\chi^2$-divergence \cite{PW_2025} 
\begin{align}
\nonumber
    K(t)
    & := 
    \int_{S\x \mR^n}
    \frac{(f_t(x)-f_*(x))^2}{f_*(x)}
    \rd x\\    
\label{chi2}
    & =
    R(t)
    +
    \int_{S\x \mR^n}
    (f_*(x)-2f_t(x))
    \rd x
    =
    R(t)-1.  
\end{align}
It is expressed in terms of an auxiliary function of time 
\begin{equation}
\label{R}
    R(t)
    :=
    \bE \theta_t(X_t)
    =
    \int_{S\x \mR^n}
    \frac{f_t(x)^2}{f_*(x)}
    \rd x
    =
    \|r_t\|^2   , 
\end{equation}
where 
\begin{equation}
\label{bR}
    \ln R(t)
    =
    \bR(f_t\| f_*) \> 0
\end{equation} 
is Renyi's second-order relative entropy \cite{R_1961} of $f_t$ with respect to $f_*$, which is an upper bound for the Kullback-Leibler relative entropy 
$$
    \bD(f_t\| f_*) = \bE \ln \theta_t(X_t) \< \bR(f_t\| f_*)
$$
(by Jensen's inequality) 
and vanishes if and only if $f_t = f_*$ in $S\x \mR^n$; cf. (\ref{bHalf}).   Here, an auxiliary  function $\theta_t: S \x \mR^n\to \mR_+$ is defined along with functions $\xi_t: S\to \mR_+$ and $\eta_t: S\x\mR^n \to \mR_+$ by  the  PDF ratios
\begin{align}
\label{theta}
  \theta_t(x)
  & :=
  \frac{f_t(x)}{f_*(x)} = \xi_t(q)\eta_t(x),\\
\label{xi}
  \xi_t(q)
  & :=
  \frac{g_t(q)}{g_*(q)},\\
\label{eta}
  \eta_t(x)
  & :=
  \frac{h_t(p\mid q)}{h_*(p\mid q)},  
\end{align}
which use (\ref{xqp}) and the factorisations (\ref{fght}), (\ref{fgh*}).  
Also, the norm $\|\cdot \|$ in (\ref{R}) is associated with the inner product $\bra \cdot, \cdot\ket$ in the Hilbert space $L^2(S \x \mR^n, \mR)$ (or the spaces $L^2(S, \mR)$ and $L^2(\mR^n, \mR)$ in what follows, depending on the context) and is applied to a function 
\begin{equation}
\label{r}
    r_t(x) 
    := 
    \frac{f_t(x)}{\psi(x)}
    =
    \kappa_t(q)
    \rho_t(p\mid q), 
    \quad
    t\> 0,\
    x \in S\x \mR^n, 
\end{equation}
where $\psi$ is another auxiliary function associated with (\ref{f*}) as 
\begin{equation}
\label{psi}
    \psi(x):= 
    \sqrt{f_*(x)}
    =
    \frac{1}{\sqrt{\Pi(\beta)}}
    \re^{-\frac{1}{2}\beta H(x)} .  
\end{equation}
The functions $\kappa_t: S \to \mR_+$ and $\rho_t(\cdot \mid q): \mR^n\to \mR_+$,  with $q\in S$, which are used in (\ref{r}), also arise from the factorisations (\ref{fght}), (\ref{fgh*}):
\begin{align}
\label{kappa}
  \kappa_t(q)
  & := \frac{g_t(q)}{\sqrt{g_*(q)}},\\
\label{rho}
  \rho_t(p\mid q)
  & := 
  \frac{h_t(p\mid q)}{\sqrt{h_*(p\mid q)}}. 
\end{align}
These factorisations allow the $\chi^2$-divergence (\ref{chi2}) to be decomposed as 
\begin{equation}
\label{KKK}
    K(t) 
    = 
    K_1(t) + K_2(t) 
\end{equation}
into a position part $K_1$ and a momentum part $K_2$ defined by 
\begin{align}
\label{K1}
    K_1(t)
    & := 
      \bE \xi_t(Q_t)-1
      =
      \|\kappa_t\|^2-1,\\
\nonumber
      K_2(t)
      & := 
          \bE
          (
    \xi_t(Q_t)
    (
    \bE 
    (
        \eta_t(X_t)
    \mid  
    Q_t
    )
    -1
    )
    )\\
\nonumber
    & =  
    \bE
    (
    \xi_t(Q_t)
    (\|\rho_t(\cdot\mid Q_t)\|^2-1)
    )\\    
\label{K2}
    & = 
    \int_S
    \kappa_t(q)^2
    (\|\rho_t(\cdot\mid q)\|^2-1)
    \rd q. 
\end{align}
Here,  the functions $\kappa_t$, $\rho_t$ from  (\ref{kappa}), (\ref{rho}) have been used along with the  PDF ratios (\ref{theta})--(\ref{eta}) and the relations 
\begin{align*}
    \bE \xi_t(Q_t)
    & =
    \int_S
    \kappa_t(q)^2
    \rd q
    =
    \re^{\bR(g_t\|g_*)},\\
    \bE 
    (
        \eta_t(X_t)
        \mid 
        Q_t=q
    )
    & =
    \int_{\mR^n}
    \rho_t(p \mid q)^2
    \rd p\\
    & =
    \re^{\bR(h_t(\cdot\mid q)\|h_*(\cdot\mid q))}  
\end{align*}
which involve the corresponding Renyi's  relative entropies. 
Note that both functions $K_1$ and $K_2$ in (\ref{K1}), (\ref{K2}) take nonnegative values. Furthermore, $K_1(t)=0$ if and only if the position $Q_t$ has the invariant PDF $g_*$ from (\ref{g*}), that is, $g_t = g_*$ everywhere in $S$.  Similarly, and in view of (\ref{gtpos}), the equality  $K_2(t)=0$ holds if and only if the momentum $P_t$, given $Q_t=q$,   has the invariant conditional momentum PDF $h_*(\cdot\mid q)$ from (\ref{h*}) for all $q\in S$, that is, when $h_t = h_*$ everywhere in $S \x \mR^n$. The latter property implies that the inequality 
\begin{equation}
\label{K>K1}
    K(t) 
    \>  
    K_1(t)
\end{equation}
(following from (\ref{KKK}) due to $K_2\> 0$) 
holds as an equality if and only if $h_t = h_*$. 

For what follows,  in order to secure finiteness and smoothness of (\ref{R}) with respect to  time, we assume that the function $r_t$ in (\ref{r}) is square integrable, that is, $r_t \in L^2(S\x \mR^n,\mR_+)$ for any $t\> 0$, and the resulting map $t\mapsto r_t$ is assumed to be infinitely differentiable as a map from $\mR_+$ to the Hilbert space  $L^2(S\x \mR^n,\mR)$: 
\begin{equation}
\label{rCinf}
    r_\bullet \in C^\infty(\mR_+, L^2(S\x \mR^n,\mR)). 
\end{equation}
This smoothness property is inherited by the function (\ref{kappa}), so that 
\begin{equation}
\label{kappaCinf}
    \kappa_\bullet \in C^\infty(\mR_+, L^2(S,\mR)),  
\end{equation}
with its time derivatives (including the function itself) being related to those of (\ref{r}) by 
\begin{align}
\nonumber
    \d_t^k \kappa_t(q)
    & = 
    \int_{\mR^n}
    \sqrt{h_*(p\mid q)}\, 
    \d_t^k r_t(q,p)
    \rd p\\
\label{kr}
    & =
    \bra
        \sqrt{h_*(\cdot\mid q)},\, 
        \d_t^k r_t(q,\cdot)
    \ket
\end{align}
for all     $t\> 0$, 
    $q \in S$ and  
$k = 0,1,2,\ldots$. Furthermore, application of the Cauchy-Bunyakovsky-Schwarz inequality to (\ref{kr}) yields 
$$
    \|\d_t^k \kappa_t\|^2
    \< 
    \int_S
        \underbrace{\|\sqrt{h_*(\cdot \mid q)}\|^2}_1
        \|\d_t^k r_t(q,\cdot)\|^2
    \rd q
    =
    \|\d_t^k r_t\|^2, 
$$ 
where use is also made of the normalisation property for the invariant conditional momentum PDF  $h_*(\cdot \mid q)$ from (\ref{h*}).

\begin{thm}
\label{th:Kdot}
Suppose the conditions of Theorem~\ref{th:gamdot} and the damping-diffusion relation (\ref{FG}) are satisfied along with (\ref{rCinf}).  
Then the $\chi^2$-divergence (\ref{chi2}) is a nonincreasing function of time:
\begin{equation}
\label{Kdot0}
  \dot{K}(t)\< 0 
\end{equation}
for any $t\> 0$. The inequality (\ref{Kdot0}) holds as an equality at those and only those moments of time when the conditional momentum PDF (\ref{ht}) coincides with its invariant counterpart (\ref{h*}), that is,  $h_t= h_*$ everywhere in the phase space $S \x \mR^n$. 
\end{thm}
\begin{proof}
With $r_t$ in (\ref{r}) inheriting from $f_t$ the infinite differentiability in time, the partial time derivatives  $\d_t^k r_t$ provide the corresponding time derivatives for $r_t$ as a map in (\ref{rCinf}). In particular, the time derivative of the function (\ref{chi2}) takes the form
\begin{equation}
\label{Kdot}
    \dot{K}(t) = \dot{R}(t) = 2\bra r_t, \d_t r_t\ket ,  
\end{equation}
where $R$ is given by (\ref{R}). A combination of (\ref{r}) with the FPKE (\ref{fdotbet}) yields
\begin{equation}
\label{rdot}
    \d_t r_t 
     = 
    \frac{1}{\psi}\d_t f_t
    =
    \frac{1}{\psi}
    \cL^\dagger(f_t).  
\end{equation}
We will now compute the right-hand side of (\ref{rdot}) by using the representation $f_t = \psi r_t$ which follows from (\ref{r}), so that 
\begin{align}
\nonumber
    \cL^\dagger(f_t)
    = &    
    \{H,\psi r_t\}
    +
        \frac{1}{2}
    \bra
        G, \d_p^2 (\psi r_t)
    \ket_\rF    \\
\label{fdotpsir} 
    & + 
    \frac{\beta}{2}
    (
    p^\rT M^{-1} G \d_p (\psi r_t)
    +
        \bra
        G, M^{-1}
    \ket_\rF 
    \psi r_t
    ) 
\end{align}
in view of (\ref{fdotbet})
regardless of a  particular structure of the function $\psi$.   
The derivation property \cite{A_1989} of the Poisson bracket (\ref{Poiss}) leads to 
\begin{align}
\nonumber
    \frac{1}{\psi}
    \{H,\psi r_t\}
    & =
    \frac{1}{\psi}
    (\{H,\psi\} r_t + \psi \{H,r_t\})\\
\label{Hrpsi}
    & =
    \{H,\lambda\}r_t + \{H,r_t\} 
\end{align}
since $\frac{1}{\psi}\{H,\psi\} = \{H,\lambda \}$,  
where $\lambda: S\x \mR^n\to \mR$ is an auxiliary function given by 
\begin{equation}
\label{lnpsi}
  \lambda
  := 
  \ln \psi . 
\end{equation}
Furthermore, by using the logarithmic gradient $\frac{1}{\psi} \d_p \psi = \d_p \lambda$ of the function $\psi$ over the momentum variables, it follows that 
\begin{align}
\label{dpsir}
    \frac{1}{\psi}\d_p (\psi r_t) 
    & = r_t \d_p \lambda  + \d_p r_t,\\
\label{ddpsir}
    \frac{1}{\psi}
    \d_p^2 (\psi r_t) 
    & = \frac{r_t}{\psi}\d_p^2 \psi  + \d_p\lambda \d_p r_t^\rT + \d_pr_t \d_p \lambda^\rT + \d_p^2 r_t, 
\end{align}
where
\begin{equation}
\label{ddpsir1}
    \frac{1}{\psi}\d_p^2 \psi
    =
    \d_p^2 \lambda + \d_p \lambda \d_p \lambda^\rT.   
\end{equation}
Substitution of (\ref{Hrpsi}) and (\ref{dpsir})--(\ref{ddpsir1}) into the right-hand side of (\ref{fdotpsir}) leads to 
\begin{align}
\nonumber
    \frac{1}{\psi}\cL^\dagger&(f_t)
    =       
    \{H,\lambda\} r_t + \{H,r_t\}\\
\nonumber
    & +
        \frac{1}{2}
    \Bra
        G, 
        \frac{r_t}{\psi} \d_p^2\psi
          + \d_p\lambda \d_p r_t^\rT + \d_pr_t \d_p \lambda^\rT+ \d_p^2 r_t
    \Ket_\rF    \\
\nonumber
    & + 
    \frac{\beta}{2}
    (
    p^\rT M^{-1} G (r_t \d_p \lambda  + \d_p r_t)
    +
        \bra
        G, M^{-1}
    \ket_\rF 
    r_t
    )\\
\nonumber
     = &     
    \{H,\lambda\} r_t + \{H,r_t\}
    + 
    \frac{1}{2}
    \bra G, \d_p^2 r_t\ket_\rF \\
\nonumber
    & +
    \Big(\d_p \lambda+\frac{\beta}{2}M^{-1}p\Big)^\rT G \d_p r_t
    \\
\nonumber
     & +
        \frac{1}{2}
    \Bra
        G, 
        \frac{1}{\psi} \d_p^2\psi
        +\beta(I_n + \d_p\lambda p^\rT) M^{-1}
    \Ket_\rF r_t   \\
\nonumber
     = &     
    \{H,\lambda\} r_t + \{H,r_t\}
    + 
    \frac{1}{2}
    \bra G, \d_p^2 r_t\ket_\rF \\
\nonumber
    & +
    \Big(\d_p \lambda+\frac{\beta}{2}M^{-1}p\Big)^\rT G \d_p r_t
    \\    
\label{fdot/psi}
     & +
        \frac{1}{2}
    \Bra
        G, 
        \d_p^2 \lambda + \d_p \lambda \d_p \lambda^\rT
        +\beta(I_n + \d_p\lambda p^\rT) M^{-1}
    \Ket_\rF r_t   . 
\end{align}
We will now take into account the special structure of $\psi$ in (\ref{psi}) as a smooth function of the Hamiltonian $H$. This property is inherited by the function $\lambda$ in (\ref{lnpsi}) which takes the form 
\begin{equation}
\label{lamH}
  \lambda = -\frac{1}{2}(\ln \Pi(\beta) + \beta H).  
\end{equation}
Since any such function commutes with $H$ in the sense of the Poisson bracket \cite{A_1989}, then  
\begin{equation}
\label{Hlam}
  \{H,\lambda \} =0. 
\end{equation}
The gradient and the Hessian matrix of the function (\ref{lamH}) over the momentum variables are found by using (\ref{qdot}), (\ref{dHdp}), (\ref{Hpp}) as
\begin{align}
\label{lamp}
    \d_p \lambda 
    & = 
    -\frac{\beta}{2} \d_p H = -\frac{\beta}{2} M^{-1} p,\\
\label{lampp}
    \d_p^2 \lambda 
    & = -\frac{\beta}{2} \d_p^2 H = -\frac{\beta}{2} M^{-1}. 
\end{align}
Substitution of (\ref{Hlam})--(\ref{lampp}) into the right-hand side of (\ref{fdot/psi}) allows the time derivative in (\ref{rdot}) to be represented in the form
\begin{align}
\nonumber
    \d_t r_t 
    = &       
    \{H,r_t\}
    + 
    \frac{1}{2}
    \bra G, \d_p^2 r_t\ket_\rF\\
\nonumber
     & +
        \frac{\beta}{4}
    \Bra
        G, 
        M^{-1}  -\frac{\beta}{2} M^{-1} p p^\rT M^{-1}
    \Ket_\rF 
    r_t\\
\label{rdot1}
    = &
    \{H,r_t\}
    +
    \frac{1}{2}
    (\Lambda r_t - \cA(r_t)).         
\end{align}
Here, use is made of an auxiliary function $\Lambda: S \to \mR_+$ given by 
\begin{equation}
\label{Lambda}
  \Lambda(q):=         
  \frac{\beta}{2}
  \bra
            G(q), M(q)^{-1}
        \ket_\rF.  
\end{equation}
Also, $\cA$ in (\ref{rdot1}) is a self-adjoint operator acting on a function $\varphi \in C^2(S\x \mR^n, \mR)$ as 
\begin{equation}
\label{cA}
  \cA (\varphi)(x)
  := 
  \frac{\beta^2}{4}
  \|M(q)^{-1}  p\|_{G(q)}^2
  \varphi
  -
  \bra
    G(q),
    \d_p^2\varphi
  \ket_\rF. 
\end{equation}
Therefore, the relation (\ref{rdot1}) can be represented in the form
\begin{equation}
\label{rdot2}
    \d_t r_t = 
    (
        \ad_H - \cB
        )
        (r_t). 
\end{equation}
Its right-hand side involves the skew self-adjoint operator $\ad_H(\cdot) := \{H,\cdot\}$ in the sense of the space $L^2(S\x \mR^n, \mR)$  and a self-adjoint operator $\cB$,  which is associated with (\ref{Lambda}), (\ref{cA}) and acts as
\begin{equation}
\label{cB}
  \cB(\varphi) := \frac{1}{2}(\cA(\varphi)-\Lambda \varphi). 
\end{equation}
The operator $\ad_H = -\ad_H^\dagger$ 
does not contribute to the quadratic form 
\begin{equation}
\label{Kdot1}
    \dot{K}(t)
    =   
    2
    \overbrace{\bra
        r_t,
        \ad_H(r_t)
    \ket}^0
    -
    2
    \bra
        r_t,
        \cB(r_t)
    \ket
    =
    -2 
    \bra
        r_t,
        \cB(r_t)
    \ket    
\end{equation}
obtained by combining (\ref{rdot2}) with (\ref{Kdot}). 
Note that for any fixed but otherwise arbitrary position $q\in S$, the operator $\cA = \cA_q$ in (\ref{cA})   acts  over the momentum variables $p_1, \ldots, p_n$. By regarding the latter as spatial  coordinates in $\mR^n$, it can be seen that the operator $\cA_q$ is isomorphic to the Hamiltonian of a multimode quantum harmonic oscillator \cite{S_1994} with the stiffness matrix $\frac{1}{2} \beta^2 M(q)^{-1}G(q)M(q)^{-1}$ and mass matrix $\frac{1}{2} G(q)^{-1}$ (see Appendix~\ref{sec:ground}). By the properties of such Hamiltonians \cite{RS_1978}, the  operator $\cA_q$ has a purely discrete spectrum,  with (\ref{Lambda}) being its lowest eigenvalue from (\ref{cAeig}):
\begin{equation}
\label{ground}
    \lambda_{\min}(\cA_q) 
    = 
    \Lambda(q). 
\end{equation}
The corresponding eigenspace of $\cA_q$  is one-dimensional and is spanned by the function $\sqrt{h_*(\cdot\mid q)}$ (as the ground state wave function for the oscillator) associated with the invariant conditional momentum PDF (\ref{h*}). This eigenspace is the null space $\ker \cB_q$  of the corresponding positive semi-definite operator $\cB = \cB_q$ in (\ref{cB}). The latter  also has a purely discrete spectrum and satisfies $\lambda_{\min}(\cB_q) = 0$ in view of (\ref{ground}), with 
\begin{equation}
\label{kercB}
  \cB_q(\sqrt{h_*(\cdot\mid q)}) = 0.   
\end{equation}
Also note that the operator $\cB$ commutes with the multiplication by a momentum-independent function:
\begin{equation}
\label{cBsig}
    \cB(\sigma\varphi) = \sigma \cB(\varphi)
\end{equation}
for any function $S \ni q \mapsto \sigma(q) \in \mR$ of position variables. Therefore, 
in view of the factorisation of $r_t$ into $\kappa_t$, $\rho_t$  in (\ref{r}), it follows from (\ref{ground}), (\ref{cBsig}) that (\ref{Kdot1}) admits the representation
\begin{align}
\nonumber
    \dot{K}(t)
    = &
    -2\bra
        \kappa_t\rho_t,
        \cB(\kappa_t\rho_t)
    \ket
    =
    -2\bra
        \kappa_t\rho_t,
        \kappa_t\cB(\rho_t)
    \ket
    \\
\label{Kdot2}
    = & 
    -2\int_S
    \kappa_t(q)^2
    \bra
        \rho_t(\cdot\mid q),
        \cB_q(\rho_t(\cdot\mid q))
    \ket
    \rd q
    \< 0,   
\end{align}
which establishes (\ref{Kdot0}). 
The above mentioned structure of the ground eigenspace of the operator $\cA_q$ (or, equivalently, $\ker \cB_q$) implies that (\ref{Kdot2}) holds as an equality if and only if there exists a $p$-independent scalar function $S \ni q\mapsto \sigma(q)>0$ such that $\rho_t(p\mid q) = \sigma(q)\sqrt{h_*(p\mid q)}$ for all $q\in S$ and $p \in \mR^n$. The latter, in view of (\ref{rho}),  is equivalent to  $\sigma$ being identically equal to 1, so that $h_t = h_*$, thus proving the second statement of the theorem. 
\end{proof}

The nonstrict inequality (\ref{Kdot0})  makes the $\chi^2$-divergence $K(t)$   applicable as a Lyapunov functional for the convergence of the PDF  $f_t$ to $f_*$, as $t\to +\infty$. However, as also established in Theorem~\ref{th:Kdot}, the strict version $\dot{K}(t) < 0$ of this inequality does not hold at those moments of time when $h_t = h_*$. Such  ``breaks'' in the dissipation of $K$ are closely related to an interplay between its position and momentum components $K_1$, $K_2$ in (\ref{KKK})--(\ref{K2}),  which will be discussed in Section~\ref{sec:break}. 

Also note that, in contrast to $K$, the functions $K_1$, $K_2$ are not monotonic with respect to  time.  In particular,  the time derivative of $K_1$, computed below for subsequent purposes,  is organised as a bilinear (rather than semi-definite quadratic) form. 

\begin{lem}
\label{lem:cGdot}
Under the assumptions of Theorem~\ref{th:Kdot}, the position component (\ref{K1}) of the $\chi^2$-divergence (\ref{chi2}) satisfies
\begin{equation}
\label{K1dot}
  \dot{K}_1(t)
  =
  2\bE (\gamma_t(Q_t)^\rT  M(Q_t)^{-1}\d_q\xi_t(Q_t) ), 
\end{equation}
where $\gamma_t$ is the conditional momentum  mean from (\ref{EPQ}), and  $\xi_t$ is the position PDF ratio (\ref{xi}).  
\end{lem}

\begin{proof}
In view of (\ref{kappaCinf}), the time derivative of (\ref{K1}) can be computed similarly to (\ref{Kdot}) as
\begin{equation}
\label{K1dotkappa}
    \dot{K}_1(t)
    =
    2
    \bra 
        \kappa_t, 
        \d_t \kappa_t
    \ket 
    =
    2
    \bra 
        \xi_t , 
        \d_t g_t
    \ket 
\end{equation}
using the identity $
    \kappa_t \d_t \kappa_t =         
    \frac{g_t}{\sqrt{g_*}}  
    \frac{\d_t g_t}{\sqrt{g_*}}  = \xi_t \d_t g_t
$ which follows from the definitions  of $\kappa_t$, $\xi_t$ in (\ref{kappa}), 
(\ref{xi}). In view of the PDE (\ref{gdot}), 
\begin{align*}
    \xi_t \d_t g_t 
    & = 
    - \xi_t \div_q(g_t M^{-1}\gamma_t)\\
    & =
    -\div_q(\xi_t g_t M^{-1}\gamma_t)
    +
    g_t \gamma_t^\rT M^{-1}\d_q \xi_t .  
\end{align*}
Hence, integration by parts on the right-hand side of (\ref{K1dotkappa}) leads to 
$$
    \dot{K}_1(t) 
    =
    2
    \int_S
    g_t(q) \gamma_t(q)^\rT M(q)^{-1}\d_q \xi_t(q)
    \rd q, 
$$
which is identical to (\ref{K1dot}) since $g_t$ is the PDF of the position vector $Q_t$. 
\end{proof}

\section{$\chi^2$-Divergence Dissipation Break Times}
\label{sec:break}

As mentioned above, the (nonstrict) monotonicity of the $\chi^2$-divergence $K$ in (\ref{chi2}), or Renyi's  relative entropy (\ref{bR}),  as a function of time, which is established in Theorem~\ref{th:Kdot}, is not shared by its position and momentum  components $K_1$, $K_2$ in (\ref{KKK})--(\ref{K2}). Nevertheless, it makes them  affect each other. In particular, since
\begin{equation}
\label{KKKdot}
    \dot{K}_1 + \dot{K}_2 = \dot{K} \< 0,
\end{equation}
then the  time derivatives $\dot{K}_1$, $\dot{K}_2$ cannot be simultaneously positive. Hence,  
an increase in one of the quantities $K_1$, $K_2$ (for example, $\dot{K}_1>0$) causes a decrease in the other (respectively, $\dot{K}_2< 0$) as if there were an exchange of Renyi's relative entropy between them. 
This ``waterbed'' phenomenon is particularly noticeable (and manifests itself through higher-order time derivatives of $K$, $K_1$, $K_2$)  in the vicinity of those moments of time when the $\chi^2$-divergence dissipation rate vanishes. 
The presence of such \emph{break times}, which, in view of Theorem~\ref{th:Kdot}, form the set 
\begin{align}
\nonumber
    \fS
    & :=
    \{t>0:\ \dot{K}(t)=0\}\\
\nonumber
    & = 
    \{t>0:\ h_t=h_*\ {\rm everywhere\ in}\ S\x \mR^n\}\\
\label{fS}
    & = 
    \{t>0:\ K_2(t)=0\}, 
\end{align}
does not allow $K$ to be a strict Lyapunov functional (in the usual sense) for the convergence of the DSH system to its equilibrium measure. A qualitative discussion   of the local behaviour of $K$, $K_1$, $K_2$ and the PDFs $g_t$, $h_t$ in the vicinity of break times is provided below similarly to \cite[Lemma~6.3]{V_2023_AHP}.

\begin{lem}
\label{lem:fS}
Under the conditions of Theorem~\ref{th:Kdot}, at any break time from the set (\ref{fS}),  the functions $K$, $K_1$, $K_2$ in (\ref{chi2}), (\ref{K1}), (\ref{K2}) satisfy
\begin{align}
\label{FFsign}
    \ddot{K}(t)
    & = 0,
    \quad
    \dddot{K}(t)\< 0,\\
\label{GGsign}
    \dot{K}_1(t)
    & = 0,
    \quad
     \ddot{K}_1(t)=
    -\ddot{K}_2(t) \<  0, \\
\label{HHsign}
    \dot{K}_2(t) & = 0,
    \quad
    \ddot{K}_2(t) \> 0,
    \qquad
    t \in \fS.
\end{align}
\end{lem}
\begin{proof}
At any break time $t\in \fS$ from (\ref{fS}), the function $\dot{K}$ (which is nonpositive everywhere  by (\ref{Kdot0}) of Theorem~\ref{th:Kdot})   achieves its global maximum  value $0$. This implies (\ref{FFsign}) for the first two time derivatives of $\dot{K}$ as the first and second-order necessary conditions of the maximum. 
Furthermore, (\ref{fS}) implies that  for any $t \in \fS$, the nonnegative function $K_2$ (which inherits smoothness from $K$, $K_1$ in view of (\ref{KKK})) achieves its global minimum value $0$, whereby its first two time derivatives satisfy (\ref{HHsign}) as the first and second-order necessary conditions of the minimum. For any $t \in \fS$, the inequality in (\ref{KKKdot}) becomes  an equality,  and hence, $\dot{K}_1 = - \dot{K}_2 = 0$ in view of the equality in (\ref{HHsign}). By a similar reasoning, the relation $\ddot{K}_1 = \ddot{K}- \ddot{K}_2$, which follows from (\ref{KKK}), leads to $\ddot{K}_1(t) = -\ddot{K}_2(t)\< 0$ in (\ref{GGsign}) for any $t\in \fS$ due to the equality in (\ref{FFsign}) combined with the inequality from (\ref{HHsign}).
\end{proof}

Since at every break time $t \in \fS$ from (\ref{fS}), the conditional momentum PDF $h_t$ coincides with its invariant counterpart $h_*$ in (\ref{h*}), then  the corresponding conditional momentum mean and covariance matrix from (\ref{EPQ}), (\ref{covPQ}) acquire the form
\begin{equation}
\label{EcovPQ*}
  \gamma_t(q) =0,
  \qquad
  \Sigma_t(q) = \cT M(q),
  \qquad
  t \in \fS,\
  q \in S. 
\end{equation}
In particular, the fulfillment of the property $\gamma_t=0$  at any such time makes $\d_t g_t$ for the position PDF $g_t$ in (\ref{gdot}) also vanish:  
\begin{equation}
\label{gdot*}
  \d_t g_t = 0,
  \qquad
  t \in \fS. 
\end{equation}
Moreover, for any $t\in \fS$, not only $\dot{K}(t) = 0$ (as obtained in Theorem~\ref{th:Kdot}), but also the mean energy dissipation rate vanishes:
\begin{align*}
\nonumber
    \frac{\rd }{\rd t}\bE H(X_t)
      = &
    \bE 
    \Big(
    \frac{1}{2}
    \bra
        G(Q_t),
        M(Q_t)^{-1}
    \ket_\rF  \\
\nonumber
    & -
    \|M(Q_t)^{-1}P_t\|_{F(Q_t)}^2
    \Big)\\
\nonumber
      = &
    \bE 
    \Big(
    \frac{1}{2}
    \bra
        G(Q_t),
        M(Q_t)^{-1}
    \ket_\rF  \\
\nonumber
    & -
    \bra
        F(Q_t), 
        M(Q_t)^{-1} 
        \underbrace{\bE (P_tP_t \mid Q_t)}_{\cT M(Q_t)}
        M(Q_t)^{-1} 
    \ket_\rF
    \Big)\\ 
\label{EHdot2}
    = & 
    \bE
    \Bra
        \frac{1}{2}
        G(Q_t)
        -
        \cT F(Q_t), 
        M(Q_t)^{-1} 
    \Ket_\rF
    =0 
\end{align*}
in view of (\ref{Qdot}), (\ref{EHdot1}), (\ref{FG}) and both equalities in (\ref{EcovPQ*}).  Nevertheless, as discussed below, at any break time $t \in \fS$,  the PDF $f_t$ of the position-momentum vector $X_t$ keeps evolving unless $g_t= g_*$.

\begin{lem}
\label{lem:rdot*}
Under the conditions of Theorem~\ref{th:Kdot}, at any break time in (\ref{fS}),  the time derivative of the function (\ref{r}) takes the form 
\begin{equation}
\label{rdot*}
  \d_t r_t(x) 
  =
  -\psi(x)
  p^\rT M(q)^{-1} \d_q \xi_t(q),
  \qquad
  t \in \fS, 
\end{equation}
for all $q\in S$, $p\in \mR^n$ and $x \in S\x \mR^n$ from (\ref{xqp}),  with $\psi$ the function  given by (\ref{psi}), and $\xi_t$ the position PDF ratio from (\ref{xi}). 
\end{lem}

\begin{proof}
As discussed in the proof of Theorem~\ref{th:Kdot}, for any break time $t \in \fS$, the function (\ref{rho}) takes the form $\rho_t = \sqrt{h_*}$, and hence, the corresponding function $r_t$ in (\ref{r}) satisfies  
\begin{equation}
\label{cBr0}
    \cB(r_t) = \kappa_t\cB(\rho_t) = \kappa_t \cB(\sqrt{h_*}) = 0,
    \qquad
    t \in \fS
\end{equation}
in  view of (\ref{kercB}), (\ref{cBsig}), so that $r_t \in \ker \cB$.  This reduces the time derivative (\ref{rdot2}) to 
\begin{align}
\nonumber
  \d_t r_t 
  & = 
  \ad_H(r_t)
  =
  \ad_H(\kappa_t \sqrt{h_*})\\
\nonumber
  & =
  \ad_H(\xi_t \psi)
  =
  \ad_H(\xi_t) \psi + \xi_t  \ad_H(\psi)\\
\label{rdot*1}
  & =
  \ad_H(\xi_t)\psi, 
  \qquad
  t \in \fS, 
\end{align}
where use is also made of the commutation property $\ad_H(\psi) = 0$ of the function $\psi = \sqrt{g_* h_*}$ from (\ref{psi}). Since $\xi_t$ in (\ref{xi}) does not depend on the momentum variables, then $\d_p \xi_t = 0$, so that 
\begin{align}
\nonumber
    \ad_H(\xi_t)
    & = 
    \d_q H^\rT \d_p \xi_t - \d_p H^\rT \d_q \xi_t \\
\label{adHxi}
    & = 
    -p^\rT M^{-1} \d_q \xi_t
\end{align}
in view of (\ref{dHdp}), (\ref{qdot}). Substitution of (\ref{adHxi}) into (\ref{rdot*1}) leads to (\ref{rdot*}). 
\end{proof}

From (\ref{rdot*}), it follows that, at any break time $t \in \fS$ (that is, when $h_t = h_*$),  the time derivative $\d_t h_t$ of the conditional momentum PDF $h_t$  vanishes everywhere in $S\x \mR^n$ if and only if $g_t=g_*$ everywhere in $S$, in which case the system state vector $X_t$ has the equilibrium PDF $f_t=f_*$.    This can be seen from the relation 
\begin{equation}
\label{rdot*2}
    \d_t r_t 
    = 
    \frac{1}{\psi}
    (h_t\d_t g_t + g_t \d_t h_t)
    =
    \frac{g_t}{\psi}
    \d_t h_t, 
    \qquad
    t \in \fS, 
\end{equation}
where (\ref{r}) is combined with the factorisation (\ref{fght}) and the property (\ref{gdot*}). Indeed, if a break time $t \in \fS$  is such that $\d_t h_t = 0$ everywhere in $S \x \mR^n$, then a combination of (\ref{rdot*}) with (\ref{rdot*2}) implies that $\d_q \xi_t = 0$ everywhere in $S$, which, by (\ref{xi}), is equivalent to $g_t=g_*$. The case of a \emph{nonequilibrium} break time $t\in \fS$ (that is, when $g_t\ne g_*$) is illustrated by Fig.~\ref{fig:gh}.  
\begin{figure}[htbp]
{\centering
\includegraphics[width=5cm]{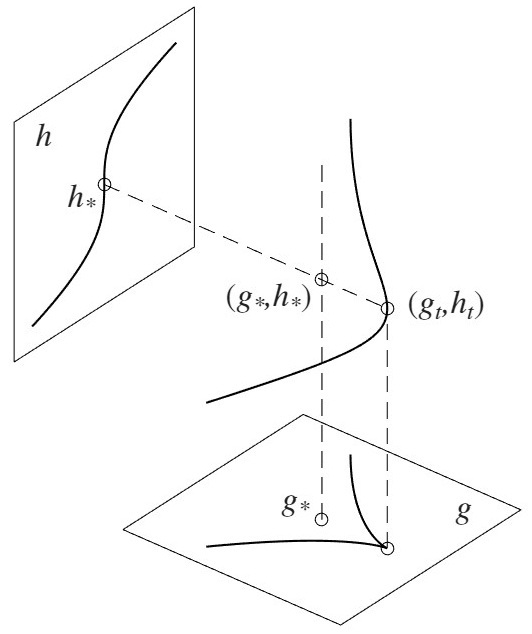}
\caption{An illustration of the local behaviour of the trajectory of the pair of the position and conditional momentum PDFs in a neighbourhood of a nonequilibrium break time $t\in \fS$,  when $h_t=h_*$ while  $g_t\ne g_*$. At any such moment of time, $\d_t g_t = 0$ according to (\ref{gdot*}), whereas $\d_t h_t\ne 0$. 
\label{fig:gh}}}
\end{figure}

The following lemma enhances the inequalities in (\ref{FFsign})--(\ref{HHsign}).

\begin{lem}
\label{lem:F...}
Under the conditions of Theorem~\ref{th:Kdot}, at any break time in (\ref{fS}),  the third and second-order  time derivatives of the $\chi^2$-divergence $K$ and its position component $K_1$  in (\ref{chi2}), (\ref{K1}) satisfy
\begin{align}
\label{K...}
  \dddot{K}(t)
  & =
  -2
  \bE_*
  (
    \|
        M^{-1}\d_q \xi_t
    \|_G^2
  ),\\
\label{K1..}
  \ddot{K}_1(t)
  & =
  -2\cT
  \bE_*
  (
    \|
        \d_q \xi_t
    \|_{M^{-1}}^2
  ),
  \qquad
  t \in \fS,
\end{align}
where $\bE_*(\cdot)$ is the expectation over the invariant position PDF $g_*$ from  (\ref{g*}), and $\xi_t$ is the position PDF ratio in (\ref{xi}). 
\end{lem}

\begin{proof}
Repeated differentiation of (\ref{Kdot}) in time leads to 
\begin{align}
\nonumber
    \dddot{K}(t) 
    & = 
    2\frac{\rd }{\rd t}
    (\|\d_t r_t\|^2 + \bra r_t, \d_t^2 r_t\ket)\\
\label{Kdddot}
    & = 
    2
    (3 \bra \d_t r_t, \d_t^2 r_t\ket + \bra r_t, \d_t^3 r_t\ket).  
\end{align}
By (\ref{rdot2}), the time derivatives of the function $r_t$ from (\ref{r}) is represented in terms of the Poisson bracket $\ad_H$ and the operator $\cB$ in (\ref{cB}) as 
\begin{equation}
\label{rdotk}
    \d_t^k r_t = 
    (
        \ad_H - \cB
    )^k
        (r_t) ,
        \qquad
        k =1, 2, 3, \ldots.
\end{equation}
At any break time, the function $r_t$ satisfies (\ref{cBr0}), which reduces (\ref{rdotk}) to   
\begin{equation}
\label{rdotk*}
    \d_t^k r_t = 
    (
        \ad_H - \cB
     )^{k-1}
        (\ad_H(r_t)), 
        \qquad
        t \in \fS . 
\end{equation}
This allows the first inner product on the right-hand side of (\ref{Kdddot}) to be computed as 
\begin{align}
\nonumber   
    \bra \d_t r_t, \d_t^2 r_t\ket
    & =   
    \bra 
        \d_t r_t, 
            (
        \ad_H - \cB
        )(\d_t r_t)
    \ket\\
\nonumber
     & =  
    \underbrace{\bra 
        \d_t r_t, 
        \ad_H(\d_t r_t)
    \ket}_0
    -
    \bra 
        \d_t r_t, 
            \cB(\d_t r_t)
    \ket    \\
\label{prod1}
    & = 
    -
    \bra 
        \d_t r_t, 
            \cB(\d_t r_t)
    \ket,     
    \qquad
    t \in \fS, 
\end{align}
where the skew self-adjointness $\ad_H = -\ad_H^\dagger$ is also used. By a similar reasoning, a combination of (\ref{rdotk*}) with the self-adjointness of the operator $\cB$ leads to the following representation of the second inner product on the right-hand side of (\ref{Kdddot}):
\begin{align}
\nonumber   
    \bra r_t, \d_t^3 r_t\ket
    & =   
    \bra 
        r_t, 
        (
            \ad_H - \cB
        )^2(\d_t r_t)
    \ket\\
\nonumber
    & = 
    -
    \bra 
        (
            \ad_H + \cB
        )(r_t), 
        (
            \ad_H - \cB
        )(\d_t r_t))
    \ket\\    
\nonumber
    & = 
    -
    \bra 
            \d_t r_t, 
        (
            \ad_H - \cB
        )(\d_t r_t)
    \ket\\        
\label{prod2}
    & =  
    \bra 
        \d_t r_t, 
            \cB(\d_t r_t)
    \ket,     
    \qquad
    t \in \fS.  
\end{align}
Therefore, substitution of (\ref{prod1}), (\ref{prod2}) into (\ref{Kdddot}) leads to
\begin{equation}
\label{Kdddot1}
  \dddot{K}(t)
  =
    -4
    \bra 
        \d_t r_t, 
            \cB(\d_t r_t)
    \ket,     
    \qquad
    t \in \fS.    
\end{equation}
In order to evaluate the operator $\cB$ in (\ref{cB}), acting over the momentum variables,  at the function $\d_t r_t$ in (\ref{rdot*}), we note that  
\begin{align*}
    \d_p^2 \d_tr_t 
    =& 
    -M^{-1} \d_q\xi_t \d_p\psi^\rT - \d_p\psi \d_q\xi_t^\rT M^{-1}\\
    &  - 
    p^\rT M^{-1} \d_q\xi_t
    \d_p^2 \psi ,
    \qquad
    t \in \fS. 
\end{align*}
Hence, in view of (\ref{cA}), (\ref{cB}) and the fact that $p^\rT M^{-1} \d_q\xi_t$ is scalar-valued, 
\begin{align}
\nonumber
    \cB(\d_t r_t) 
    & = -\cB(\psi)p^\rT M^{-1} \d_q\xi_t + \d_q \xi_t^\rT M^{-1} G \d_p\psi\\
\nonumber
    & = 
    \d_q \xi_t^\rT M^{-1} G \d_p\psi, \\
\label{cBr*}    
    & = 
    -\frac{\beta}{2}\psi \d_q \xi_t^\rT M^{-1} G M^{-1}p,
    \qquad
    t \in \fS, 
\end{align}
since the function (\ref{psi}) satisfies $\psi \in \ker \cB$ due to (\ref{kercB}),  and $\d_p \ln \psi = \d_p \lambda = -\frac{\beta}{2}M^{-1}p$ from (\ref{lnpsi}), (\ref{lamp}). Substitution of (\ref{rdot*}), (\ref{cBr*}) into (\ref{Kdddot1}) and using the property that $\psi^2 = f_*$ yields 
\begin{align*}
    \dddot{K}(t)
    = & 
    -2\beta
    \int_{S\x \mR^n}
    f_*
    \d_q \xi_t^\rT M^{-1} G M^{-1}p
    p^\rT M^{-1} \d_q \xi_t
    \rd x\\
    = &  
    -2\beta
    \int_S
    g_*(q)
    \d_q \xi_t(q)^\rT M(q)^{-1} G(q) M(q)^{-1}\\
    & \x
    \Big(
    \underbrace{
    \int_{\mR^n}
    h_*(p\mid q)
    p
    p^\rT
    \rd p}_{\cT M(q)}
    \Big)
     M(q)^{-1} 
    \d_q \xi_t(q)
    \rd q\\
    =&
    -2
    \int_S
    g_*(q)
    \|M(q)^{-1}\d_q \xi_t(q)\|_{G(q)}^2
    \rd q,
    \qquad
    t \in \fS, 
\end{align*}
in view of (\ref{EcovPQ*}), (\ref{betau}), thus establishing (\ref{K...}). We will now prove (\ref{K1..}).  From (\ref{kappa}), (\ref{gdot*}),  it follows that $\d_t \kappa_t = \frac{1}{\sqrt{g_*}} \d_t g_t = 0$  at any break time $t \in \fS$. Hence, 
differentiation of the first equality in (\ref{K1dotkappa}) yields 
\begin{align}
\nonumber
    \ddot{K}_1(t)
    & =
    2
    (\|\d_t \kappa_t\|^2
    + 
    \bra 
        \kappa_t, 
        \d_t^2 \kappa_t
    \ket 
    )\\
\label{K1ddot}
    & =
    2
    \bra 
        \kappa_t, 
        \d_t^2 \kappa_t
    \ket 
    =
    2
    \bra 
        \xi_t, 
        \d_t^2 g_t
    \ket,  
    \qquad
    t \in \fS,      
\end{align}
where use is also made of the identity $
    \kappa_t \d_t^2 \kappa_t =         
    \frac{g_t}{\sqrt{g_*}}  
    \frac{\d_t^2 g_t}{\sqrt{g_*}}  = \xi_t \d_t^2 g_t
$ in view of (\ref{kappa}), (\ref{xi}). A combination of (\ref{gdot}) with (\ref{gdot*}) implies that 
\begin{align}
\nonumber
  \d_t^2 g_t
  & = 
  - \div_q(\d_t g_t M^{-1}\gamma_t + g_t M^{-1}\d_t \gamma_t)\\
\label{gddot*}
    & = 
    - \div_q(g_t M^{-1}\d_t \gamma_t),
    \qquad
    t \in \fS. 
\end{align}
At any break time, the relations (\ref{EcovPQ*}) reduce the PDE (\ref{gamdot2}) to 
\begin{align}
\nonumber
    \d_t \gamma_t & +
    \cT \d_q \ln g_t\\
\label{gamdot*}
     = &  
      -\nabla V
  +
  \frac{\cT}{2}
  (
  \bra
    M_k,
    M
  \ket_\rF
  )_{1\< k \< n},
  \qquad
  t \in \fS. 
\end{align}
Here, 
$
    \bra M_k, M\ket_\rF = \bra M^{-1}, \d_{q_k} M\ket_\rF = \d_{q_k} \ln\det M 
$  
in view (\ref{Mk}),  and hence, the right-hand side of (\ref{gamdot*}) is related to the invariant position PDF (\ref{g*}) by
\begin{align}
\nonumber
          -\nabla V
  & +
  \frac{\cT}{2}
  (
  \bra
    M_k,
    M
  \ket_\rF
  )_{1\< k \< n}\\
\label{right1}
  & =
  \cT
  \Big(-\beta\nabla V + \frac{1}{2}
  \d_q \ln \det M\Big)
  =
  \cT \d_q \ln g_*, 
\end{align}
where (\ref{betau}) is also used. Substitution of (\ref{right1}) into (\ref{gamdot*}) leads to 
\begin{equation}
\label{gamdot*1}
    \d_t \gamma_t 
    = 
    -\cT \d_q \ln \xi_t,
    \qquad
    t \in \fS
\end{equation}
in view of (\ref{xi}). A combination of (\ref{K1ddot}) with (\ref{gddot*}), (\ref{gamdot*1}) and integration by parts yield
\begin{align*}
    \ddot{K}_1(t)
    & = 
    2\cT
    \bra 
        \xi_t, 
        \div_q(g_t M^{-1}\d_q \ln \xi_t)
    \ket \\
    & = 
    -2\cT
    \int_S
        g_t(q)
        \d_q \xi_t(q)^\rT M(q)^{-1}\d_q \ln \xi_t(q)
        \rd q    \\
    & = 
    -2\cT
    \int_S
        g_*(q)
        \|\d_q \xi_t(q)\|_{M(q)^{-1}}^2
        \rd q, 
        \qquad
        t \in \fS 
\end{align*}
(where (\ref{xi}) is used again), which establishes (\ref{K1..}) and completes the proof.  
\end{proof}

The relations (\ref{K...}), (\ref{K1..}), which involve weighted Dirichlet quadratic forms of the position PDF ratio (\ref{xi}), show that $\dddot{K}(t)< 0$ and $\ddot{K}_1(t)< 0$ (and hence, $\ddot{K}_2(t)>0$ by the second equality in (\ref{GGsign})) at any nonequilibrium break time $t\in \fS$ (that is, when $h_t = h_*$, whereas $g_t\ne g_*$).

\section{Strict Monotonicity of $\chi^2$-Divergence}
\label{sec:strict}

The following theorem, which is a corollary from Theorem~\ref{th:Kdot} and Lemmas~\ref{lem:fS}--\ref{lem:F...} and is similar to \cite[Theorem~6.6]{V_2023_AHP},  establishes strict monotonicity for the $\chi^2$-divergence (\ref{chi2}).

\begin{thm}
\label{th:entmono}
For the DSH system (\ref{Qdot}), (\ref{dP})  with the Langevin damping (\ref{uopt}) satisfying the damping-diffusion condition  (\ref{FG}),  suppose the position-momentum PDF satisfies (\ref{fsmooth}), (\ref{fgood}), (\ref{gtpos}), (\ref{rCinf}). Then the $\chi^2$-divergence (\ref{chi2}) with respect to the invariant PDF $f_*$ in (\ref{f*}) is a strictly decreasing function of pre-equilibrium time (that is, until $f_*$ is reached). 
\end{thm}

\begin{proof}
In order to eliminate from consideration the trivial situation when the system has already reached its equilibrium PDF $f_*$, we restrict the set $\fS$ of break times  in  (\ref{fS}) to
\begin{equation}
\label{fStau}
    \fS_\tau:
    =
    (0,\tau)\bigcap \fS
    =
    \{0< t< \tau:\ \dot{K}(t)=0\}. 
\end{equation}
Here, $\tau$ is the first time when the PDF $f_t$ coincides with the invariant PDF $f_*$:
\begin{equation}
\label{tau}
    \tau
    :=
    \inf
    \{
        t \> 0:\
        f_t = f_*\
        {\rm everywhere\ in}\
        S\x\mR^n
    \}, 
\end{equation}
with the convention that $\tau:= +\infty$ if the above set is empty. From (\ref{fStau}), (\ref{tau}), it follows that  for any $t \in \fS_\tau$, the position PDF is different from its invariant counterpart: $g_t \ne g_*$.  Indeed, together with $h_t =h_*$ at any $t \in \fS$,      the equality $g_t =g_*$ would imply that $f_t=f_*$ (that is, the equilibrium is already reached) and hence,  $t \> \tau$, thus contradicting the inequality $t< \tau$.  Hence, in view of (\ref{K...}),
\begin{equation}
\label{F...neg}
    \dddot{K}(t) < 0,
    \qquad
    t \in \fS_\tau.
\end{equation}
In combination with (\ref{fS}) and the equality in (\ref{FFsign}), the strict inequality (\ref{F...neg}) implies that any such $t$ is a stationary point of inflection for the $\chi^2$-divergence $K$ as a smooth function of time. In a small neighbourhood of such a point, $K$ behaves asymptotically as a strictly decreasing cubic parabola:
\begin{equation}
\label{cubic}
    K(s)=K(t) + \frac{1}{6}\dddot{K}(t)(s-t)^3 + o(|s-t|^3),
\end{equation}
as $s \to t \in \fS_\tau$. 
Therefore, the set $\fS_\tau$ in (\ref{fStau}) consists of isolated points and is countable (hence, of zero Lebesgue measure), which, in view of
\begin{equation*}
\label{Ffot...neg}
    \dot{K}(t) < 0,
    \qquad
    t \in (0,\tau)\setminus \fS,
\end{equation*}
makes $\dot{K}$ strictly negative almost everywhere in the interval $[0,\tau]$. The latter implies that $K(t)-K(s) = \int_s^t \dot{K}(u)\rd u < 0$ for all $0\< s < t \< \tau$, and hence, $K$ is a strictly decreasing function of time over $[0,\tau]$.
\end{proof}

The strict monotonicity of the $\chi^2$-divergence $K$, proved in Theorem~\ref{th:entmono},  employs the observation that the position-momentum distribution does not ``dwell'' at those pairs $(g_t,h_t)$,  where $\dot{K}(t)=0$,  unless the system has reached the equilibrium PDF $f_*$.
 As illustrated in Fig.~\ref{fig:FGH},
\begin{figure}[htbp]
{\centering
\includegraphics[width=6 cm]{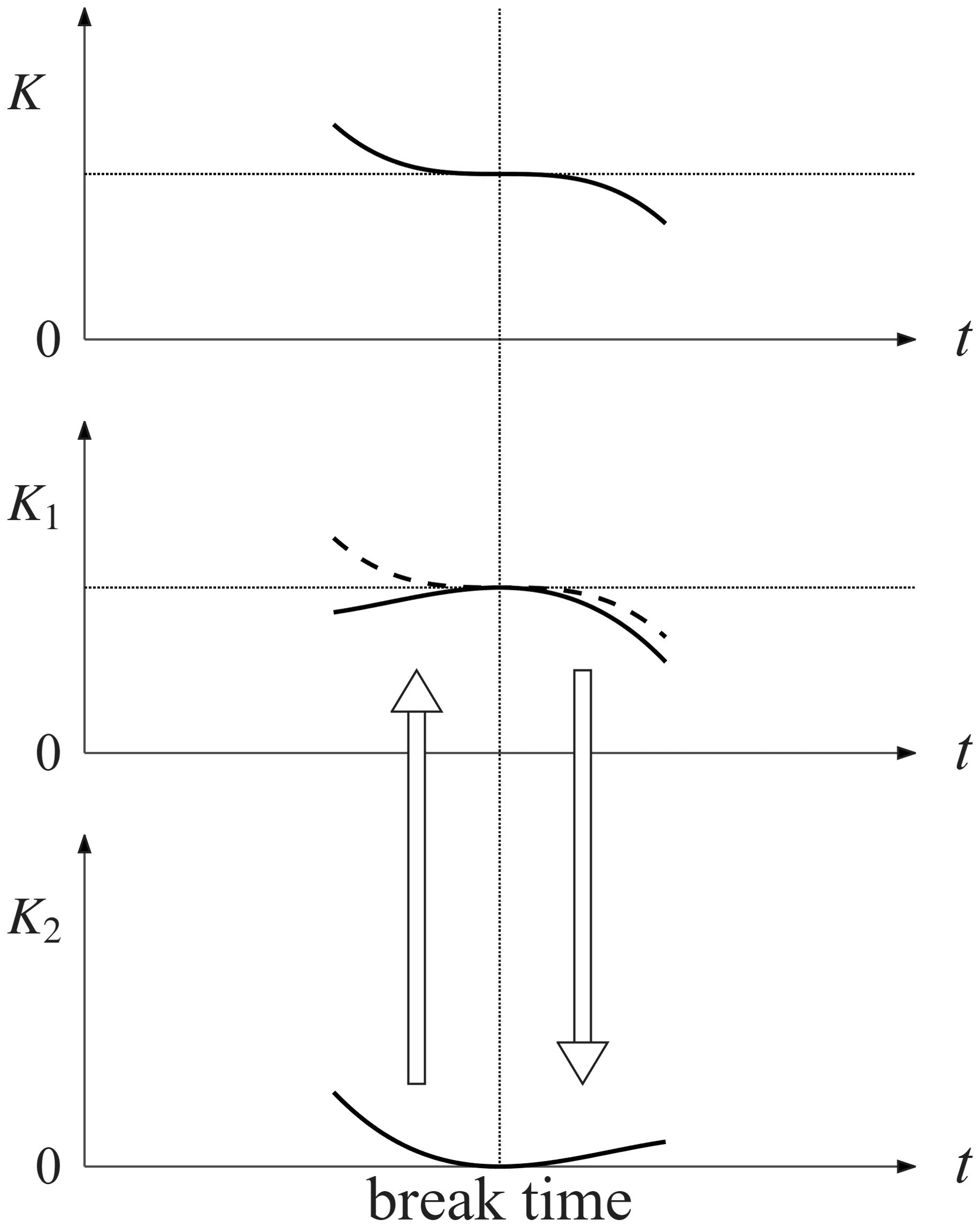}
\caption{An illustration of the local behaviour of the $\chi^2$-divergence $K$ and its position and momentum components $K_1$, $K_2$ in (\ref{chi2}), (\ref{K1}), (\ref{K2}) in the vicinity of a pre-equilibrium  break time in (\ref{fStau}). It is described asymptotically by the cubic parabola (\ref{cubic}) and two quadratic parabolas, whose nonzero leading coefficients (\ref{K...}), (\ref{K1..}) make the inequalities in 
(\ref{FFsign})--(\ref{HHsign}) strict. At any break time, $K$  and $K_1$ share a common value (the graph of $K$ is superimposed as the dashed curve on that of $K_1$); cf. (\ref{K>K1}). The white arrows indicate the direction of Renyi's relative entropy flow between the position and momentum components immediately before and after  the break time (from $K_2$ to $K_1$ and from $K_1$ to $K_2$, respectively). 
\label{fig:FGH}}}
\end{figure}
the local behaviour of $K$ and its components $K_1$, $K_2$ in the vicinity of any pre-equilibrium   break time in (\ref{fStau}) is described asymptotically by the cubic parabola (\ref{cubic}) and two (concave and convex)  quadratic parabolas, respectively.  The inequalities in (\ref{FFsign})--(\ref{HHsign}) for the leading  coefficients of these parabolas are strict  in view of (\ref{K...}), (\ref{K1..}), which can be interpreted as a local exchange between the position and momentum  components $K_1$, $K_2$, so that the increase in one of them is compensated by the decrease in the other at the level of the first and second-order terms of their local Taylor 
series expansions. Immediately before such a break, the momentum component $K_2$  flows into the position component $K_1$  until $K_2$ is  depleted at the break time,  after which the opposite flow (from $K_1$ to $K_2$) is observed over a short period of time. 

The ``waterbed''  scenario in the dissipation of Renyi's relative entropy with respect to the invariant measure, established above,  is similar to the potential-kinetic energy exchange in the total energy dissipation for deterministic dissipative Hamiltonian dynamics mentioned in Introduction. It can be interpreted as a manifestation of the BKL principle \cite{BK_1952,L_1960} in application to the $\chi^2$-divergence $K$ as a Lyapunov functional for the FPKE in the DSH setting.  A similar  behaviour was discussed for the Kullback-Leibler relative entropy case in \cite{V_2023_AHP}.

\section{Conclusion}\label{sec:conc}

For multivariable DSH systems,  described by an ODE for the position and an Ito SDE for the momentum with conservative, Langevin viscous  damping and external  random forces satisfying the damping-diffusion relation, we have applied the position-momentum conditioning from \cite{V_2023_AHP} to the $\chi^2$-divergence  with respect to the Maxwell-Boltzmann invariant measure. With the latter favouring lower energy values, the setting has been discussed from the stochastic optimisation and physical modelling viewpoints.    
We have decomposed the $\chi^2$-divergence, involving Renyi's second-order relative entropy,  into the position and momentum parts and 
shown that it has a nonpositive time derivative which vanishes only at break times when so does the momentum component.  This property has been linked to auxiliary quantum harmonic oscillator Hamiltonians parameterised by the position space of the DSH system.  Using higher-order dissipation relations, 
we have shown that all the pre-equilibrium $\chi^2$-divergence dissipation break times  are isolated  strong local maxima and minima of the position and momentum components, respectively. They compensate each other in such a way that  the break times are stationary inflection points of the $\chi^2$-divergence, making the latter (and thus, Renyi's relative entropy)  a strictly decreasing Lyapunov functional. The  position-momentum  Renyi's relative entropy exchange has been discussed as a manifestation of the BKL invariance principle in application to the FPKE in the DSH setting, similar to yet different  from the Kullback-Leibler relative entropy case of \cite{V_2023_AHP}.  
While our analysis pertains mainly to qualitative aspects of the convergence to equilibrium and does not consider the convergence rate issues (such as, for example, in \cite{HN_2004,T_2002}), the results of the paper (including the spectral gaps for the auxiliary quantum harmonic oscillator Hamiltonians) can be developed towards quantitative convergence rate estimates.

\section*{Acknowledgement}

Support from the Australian Research Council is gratefully acknowledged.

\appendix
\section{Common Morse indices for Hamiltonian and potential energy critical points}    
\label{sec:Morse}

For completeness of exposition, we also note in regard to (\ref{HVmin}), (\ref{HVargmin}) that the critical points of the  Hamiltonian $H$ in (\ref{HVT}) in the phase space $S \x \mR^n$  and the potential energy $V$ in the position space $S$ are in a one-to-one correspondence and share the Morse indices \cite{M_1963}. Indeed, by (\ref{T})--(\ref{H'}), the vector $x \in S\x \mR^n$ in (\ref{xqp}) is a critical point of $H$ (that is, $H'(x)=0$) if and only if $\nabla V(q)=0$ and $p=0$, in which case, the matrices 
\begin{align*}
    \d_q^2 T(x)
    & = 
    -\frac{1}{2}
    (p^\rT \d_{q_k} M_j(q) p)_{1\< j,k\< n},\\
    (\d_p\d_q H(x))^\rT 
    & = 
    \d_q\d_p H(x) = \d_q (M(q)^{-1}p)\\
    & = 
    -
    \begin{bmatrix}
      M_1(q)p  & \ldots & M_n(q)p 
    \end{bmatrix} 
\end{align*}
(where use is made of (\ref{Mk})) also vanish. Hence, the Hessian matrix 
\begin{align}
\nonumber
    H'' 
    & :=
    \d_x^2
    H
    =
    \begin{bmatrix}
        \d_q^2 H & \d_p\d_q H \\
        \d_q\d_p H & \d_p^2 H
    \end{bmatrix}\\
\label{H''}
    & = 
    \begin{bmatrix}
        V'' + \d_q^2 T & \d_p\d_q H \\
        \d_q\d_p H & M^{-1}
    \end{bmatrix}
\end{align}
of the Hamiltonian  reduces at any such point $x$  to a block-diagonal matrix 
$$
    H'' 
    = 
    \begin{bmatrix}
        V'' & 0\\
        0 & M^{-1}
    \end{bmatrix}. 
$$
The latter satisfies $\det H'' = \frac{\det V''}{\det M}$ (whereby $H$ inherits the Morse function property from $V$)  and has the same number of negative eigenvalues as $V''$ due to the positive definiteness of the mass matrix $M$.

\section{Expectation by parts formula}    
\label{sec:exparts}

The following lemma provides a variant of the  integration by parts under expectation, which is used in the proof of Theorem~\ref{th:gamdot}. 

\begin{lem}
\label{lem:parts}
Let $g \in C^1(S, (0,+\infty))$ be an everywhere  positive continuously differentiable PDF on the set $S$ described at the beginning of Section~\ref{sec:stopt}. Then for any map $A:= (a_{jk})_{1\< j \< r, 1\< k \< n} \in C^1(S, \mR^{r\x n})$ and a function $\varphi \in C^1(S, \mR)$ of bounded support,    
\begin{equation}
\label{exparts}
  \bE(A\nabla \varphi)
  =
  -\bE((\div A + A \nabla \ln g)\varphi), 
\end{equation}
where $\bE \psi := \int_S \psi(q) g(q)\rd q$ is the expectation of a Borel measurable function $\psi$ on $S$ over the PDF $g$, and $\div A := (\sum_{k=1}^n \nabla_k a_{jk} )_{1\< j \< r} \in C(S, \mR^r)$, with $\nabla_k(\cdot)$ the partial derivative with respect to the $k$th coordinate in $\mR^n$, $k = 1, \ldots, n$.   
\end{lem}
\begin{proof}
A vector-matrix version of the Leibnitz rule yields a 
divergence identity 
\begin{align}
\nonumber
    \div(g\varphi A)
    & =
    g A \nabla \varphi    
    +
    g\varphi \div A    
    +
    \varphi A \nabla g\\
\nonumber
    & =
    (A \nabla \varphi    
    +
    \varphi \div A
    +
    \varphi 
    A \nabla \ln g)
    g\\
\label{divgphiA}
    & =
    (A \nabla \varphi    
    +
    (\div A
    +
    A \nabla \ln g)
    \varphi)
    g,     
\end{align}
where the logarithmic gradient $\frac{1}{g}\nabla g = \nabla \ln g$ is well defined since $g>0$. By integrating both sides of (\ref{divgphiA}) over the set $S$ and using the Ostrogradsky-Gauss divergence theorem along with the assumption that $\varphi$ is of bounded support, it follows that  
$
    \bE
    (A \nabla \varphi    
    +
    (\div A
    +
    A \nabla \ln g)
    \varphi)
    =0  
$. The latter is equivalent to (\ref{exparts}). 
\end{proof}

In the case of $S = \mR^n$, the ``expectation by parts'' formula (\ref{exparts}) remains valid if  the bounded support requirement for $\varphi$ is replaced with the growth constraint $g(q)\varphi(q) A(q) = o(|q|^{1-n})$, as $q\to \infty$, for the matrix field from the left-hand side of (\ref{divgphiA}). This relaxed condition still guarantees that the vector fields which  constitute the rows of $g\varphi A$ have an asymptotically vanishing flux through a sphere in $\mR^n$ as its radius goes to infinity. 

\section{Auxiliary quantum mechanical Hamiltonian}
\label{sec:ground}

For the purposes of the proof of Theorem~\ref{th:Kdot}, 
consider  a linear operator $\cH$   which is specified by constant matrices $\Xi, \mho \in \mP_n$  and maps a complex-valued function $\varphi \in C^2(\mR^n, \mC)$ to 
\begin{equation}
\label{cH}
    \cH(\varphi)(z) 
    := 
    \frac{1}{2}\|z\|_\Xi^2 
    \varphi(z)
    - 
    \frac{1}{2}
    \bra
        \mho^{-1},
        \d_z^2 \varphi(z)
    \ket_\rF,  
\end{equation}
where $z \in \mR^n$. 
This is a self-adjoint operator on  a dense domain  in $L^2(\mR^n, \mC)$ (in particular, $\cH$ is well defined on the Schwartz space \cite{V_2002} of rapidly decreasing functions on $\mR^n$)  which describes the Hamiltonian of a quantum harmonic oscillator with the potential energy function $z \mapsto \frac{1}{2} \|z\|_\Xi^2$ and the (nonrelativistic) kinetic energy operator $\varphi \mapsto -\frac{1}{2} \bra \mho^{-1}, \d_z^2 \varphi \ket_\rF$. The latter corresponds to the quantum mechanical momentum operator $-i\d_z(\cdot)$ over the vector $z$ of position variables in $\mR^n$ (with the reduced Planck constant $\hbar = 1$ as in the atomic unit system).   Accordingly, $\Xi$, $\mho$ in (\ref{cH}) play the role of the stiffness and mass matrices, respectively.  

Now, let $U \in \mR^{n\x n}$ be an orthogonal matrix whose columns are the eigenvectors of  a real positive definite  symmetric matrix $\mho^{-1/2} \Xi \mho^{-1/2} $ with eigenvalues $\omega_1, \ldots, \omega_n>0$, so that 
\begin{equation}
\label{diag}
    \mho^{-1/2} \Xi \mho^{-1/2} = U\Omega U^\rT,
    \qquad
    \Omega
    := 
    \diag_{1\< k \< n}
    (\omega_k) . 
\end{equation}
The coordinate transformation $z \mapsto U^\rT\sqrt{\mho} z=:\zeta$ in $\mR^n$, whereby $z = \mho^{-1/2}U\zeta$ and 
$$
    \varphi(z) = \phi(\zeta),
    \quad
    \d_z^2\varphi 
    = 
    \d_z\zeta^\rT  
    \d_\zeta^2\phi
    \d_z\zeta 
    =
    \sqrt{\mho} 
    U
    \d_\zeta^2\phi
    U^\rT 
    \sqrt{\mho} 
$$
(with $\d_z\zeta = U^\rT 
    \sqrt{\mho}$ the Jacobian matrix of the linear map $z\mapsto \zeta$), 
allows (\ref{cH}) to be represented in terms of the function $\phi$ as 
\begin{equation}
\label{sH}
    \cH(\varphi)(z)
    =
    \frac{1}{2}
    \|\zeta\|_\Omega^2 \phi(\zeta) 
    - 
    \frac{1}{2} 
    \Delta 
    \phi(\zeta)  
    =: 
    \sH(\phi)(\zeta).  
\end{equation}
Here, $\Delta:= \sum_{k=1}^n \d_{\zeta_k}^2(\cdot)$ is the Laplacian over the new variables $\zeta_1, \ldots, \zeta_n$ which comprise the vector $\zeta \in \mR^n$, and $\|z\|_\Xi^2 = \|\zeta\|_\Omega^2 = \sum_{k=1}^n \omega_k \zeta_k^2$ in view of (\ref{diag}). The Hamiltonian $\sH$ in (\ref{sH}) is known to have a purely discrete spectrum \cite{RS_1978} and a one-dimensional ground eigenspace spanned by the function 
\begin{equation}
\label{phi*}
    \phi(\zeta) 
    = 
    \re^{-\frac{1}{2} \|\zeta\|_{\sqrt{\Omega}}^2}
    =
    \prod_{k=1}^n
    \re^{-\frac{1}{2}\sqrt{\omega_k} \zeta_k^2}, 
\end{equation}
which is obtained from the univariate quantum harmonic oscillator \cite{S_1994} with $n=1$  by the separation of variables and associated with the lowest eigenvalue of the Hamiltonian (\ref{sH}): 
\begin{align}
\nonumber
    \lambda_{\min}(\cH)
    & =
    \lambda_{\min}(\sH) 
    = 
    \frac{1}{2}    
    \sum_{k=1}^n
    \sqrt{\omega_k}\\
\label{Heig}
    & =
    \frac{1}{2}
    \Tr \sqrt{\Omega}
    =
    \frac{1}{2}
    \Tr \sqrt{\Xi \mho^{-1}}.  
\end{align}
Here, use is also made of the isospectrality of the matrix  $\mho^{-1/2} \Xi \mho^{-1/2}$ to  $\Xi \mho^{-1} = \sqrt{\mho}\mho^{-1/2} \Xi \mho^{-1/2} \mho^{-1/2}$. The quadratic form in (\ref{phi*}) is expressed in terms of the original vector $z$ as 
\begin{align}
\nonumber   
    \|\zeta\|_{\sqrt{\Omega}}^2 
    = 
    \|U^\rT \sqrt{\mho}z\|_{\sqrt{\Omega}}^2 
    =
    \|z\|_{\sqrt{\mho} U\sqrt{\Omega} U^\rT \sqrt{\mho}}^2,  
\end{align}
where, in view of the orthogonality of the matrix $U$ in (\ref{diag}),  
\begin{align}
\nonumber
    \sqrt{\mho} U\sqrt{\Omega} U^\rT \sqrt{\mho} 
    & =
    \sqrt{\mho} 
    U
    \sqrt{U^\rT\mho^{-1/2} \Xi \mho^{-1/2} U} 
    U^\rT 
    \sqrt{\mho} \\
\nonumber   
    & = 
    \sqrt{\mho} 
    \sqrt{\mho^{-1/2} \Xi \mho^{-1/2}} 
    \sqrt{\mho}. 
\end{align}
Therefore, up to a normalisation  constant, $\phi^2$ from  (\ref{phi*}), as a function of $z \in \mR^n$, is a Gaussian PDF with zero mean  and the covariance matrix 
\begin{equation}
\label{C}
    C
    := 
    \frac{1}{2}
    \mho^{-1/2}
    \sqrt{\sqrt{\mho}\Xi^{-1} \sqrt{\mho}}\,  
    \mho^{-1/2}. 
\end{equation}
The spectral gap for the operator $\cH$ (between its lowest and the next eigenvalues)  is $\min_{1\< k \< n}\sqrt{\omega_k}$. 
In application to the operator $\cA = \cA_q$ in (\ref{cA}) from the proof of Theorem~\ref{th:Kdot}, where it   acts over the classical momentum variables $z:=p\in \mR^n$ at a fixed $q \in S$ (so that the matrices $M(q)$, $G(q)$ are constant), this operator has the form (\ref{cH}) with the following stiffness     and mass matrices:
\begin{equation}
\label{Ximho}
    \Xi:= \frac{1}{2} \beta^2 M^{-1}GM^{-1},
    \qquad
    \mho:= 
    \frac{1}{2} G^{-1}. 
\end{equation}
Accordingly, substitution of  (\ref{Ximho}) into (\ref{Heig}) yields the smallest eigenvalue for $\cA_q$: 
\begin{align}
\nonumber
    \lambda_{\min}(\cA_q)
    & =
        \frac{1}{2}
    \Tr \sqrt{\frac{1}{2} \beta^2 M^{-1}GM^{-1} \Big(\frac{1}{2} G^{-1}\Big)^{-1}}\\
\label{cAeig}
    & = 
        \frac{\beta}{2}
    \Tr \sqrt{M^{-1}GM^{-1} G }
    = 
        \frac{\beta}{2}
        \bra
            G, M^{-1}
        \ket_\rF, 
\end{align}
which coincides with (\ref{Lambda}), thus establishing (\ref{ground}). Furthermore, by  (\ref{Ximho}), the covariance matrix (\ref{C}) takes the form 
\begin{align}
\nonumber
    C
     = & 
    \frac{1}{2}
    \sqrt{2G}
    \sqrt{
        \sqrt{\frac{1}{2} G^{-1}}
        \Big(\frac{1}{2} \beta^2 M^{-1}GM^{-1}\Big)^{-1} 
        \sqrt{\frac{1}{2} G^{-1}}}
        \sqrt{2G}\\
\nonumber   
     = & 
     \cT
     \sqrt{G}\sqrt{
        G^{-1/2}
        MG^{-1}M 
        G^{-1/2}}
    \sqrt{G} = \cT M 
\end{align}
in view of (\ref{betau}) 
and thus coincides with the conditional covariance matrix of the momentum according to the invariant Gaussian PDF $h_*$  in (\ref{h*}). 

\begin{thebibliography}{99}
\bibitem{A_1989}
V.I.Arnold, \emph{Mathematical Methods of Classical Mechanics}, 2nd Ed., Springer, New York, 1989.



\bibitem{BK_1952}
E.A.Barbashin,  and N.N.Krasovskii, On the stability of motion in the large, \textit{Dokl. Akad. Nauk SSSR}, vol.  86, no. 6, 1952, pp. 453--456.

\bibitem{BKRS_2015}
V.I.Bogachev, N.V.Krylov, M.R\"{o}ckner, and S.V.Sha\-posh\-ni\-kov,
\textit{Fok\-ker–Planck–Kolmogorov Equations}, 
American Mathematical Society,
Providence, Rhode Island, 2015.

\bibitem{CT_2006}
T.M.Cover, and J.A.Thomas, {\it Elements of Information Theory},   Wiley, Hoboken, New Jersey, 2006.

\bibitem{D_1984}
J.L.Doob, 
\textit{Classical Potential Theory and Its Probabilistic Counterpart}, 
Springer, New York, 1984. 

\bibitem{E_2008}
L.C.Evans, {\it Partial Differential Equations}, American Mathematical Society, Providence, Rhode Island, 2008.

\bibitem{HN_2004}
F.H'erau, and F.Nier, Isotropic hypoellipticity and trend to equilibrium for the
Fokker-Planck equation with a high-degree potential, \emph{ Arch. Ration. Mech. Anal.}, vol.  171, no. 2, 2004, pp. 151--218.

\bibitem{H_1986}
Wm.G.Hoover,
\emph{Molecular Dynamics},
Springer,  Berlin, 1986.


\bibitem{H_1967}
L.H\"{o}rmander, Hypoelliptic second order differential equations, \textit{ Acta Math.}, vol.  119, 1967, pp.  147--171.


\bibitem{HJ_2007}
 R.A.Horn, and C.R.Johnson, {\it Matrix Analysis}, Cambridge
University Press, New York, 2007.



\bibitem{I_2019}
M.Ilie,
Fluid-structure interaction in turbulent flows; a CFD based aeroelastic algorithm using LES,
\emph{Appl. Math. Comp.}, vol. 342, 2019, pp.  309--321.

\bibitem{KMPV_2025}
D.Kalise, L.M.Moschen, 
G.A.Pavliotis, and U.Vaes, 
A spectral approach to optimal control of the
Fokker–Planck equation
\textit{IEEE Control Syst. Lett.}, vol. 9, 2025, pp. 504--509. 

\bibitem{KS_1991}
I.Karatzas, and S.E.Shreve,
\emph{Brownian Motion and Stochastic Calculus}, 2nd Ed.,
Springer, New York, 1991.


\bibitem{KGV_1983}
S.Kirkpatrick, C.D.Gelatt, and  M.P.Vecchi, 
Optimization by simulated annealing, 
\textit{Science, New Series}, vol. 220, no. 4598, 1983, pp. 671--680.

\bibitem{K_1966}
R.Kubo,
The fluctuation-dissipation theorem,
\emph{Rep. Prog. Phys.}, vol. 29, no. 1,  1966, pp. 255--284.



\bibitem{LL_1969}
L.D.Landau, and E.M.Lifshitz, 
\textit{Statistical Physics}, 2nd Ed., 
Pergamon Press, Oxford, 1969. 

\bibitem{L_1960}
J.P.LaSalle, Some extensions of Liapunov's second method, \textit{IRE Trans. Circuit Theory}, vol. 7, 1960, pp. 520--527.


\bibitem{LPS_1996}
G.A.Leonov, D.V.Ponomarenko, and V.B.Smirnova, \textit{Frequency-Domain Methods for Nonlinear Analysis},  World Scientific, Singapore, 1996. 


\bibitem{ME_1981}
N.F.G.Martin, and J.W.England, \textit{Mathematical Theory of Entropy},
Addison-Wesley, Reading, Massachusetts, 1981.




\bibitem{M_1963} 
J.Milnor, \textit{Morse Theory},
Princeton University Press,
Princeton, New Jersey,
1963.


\bibitem{OPU_2012}
H.Ouyang, I.R.Petersen, and V.Ugrinovskii, Lagrange stabilization of pendulum-like systems: a pseudo  $H_\infty$ control approach, \emph{IEEE Trans. Autom.
Control}, vol. 57, no. 3, 2012, pp. 649--662. 


\bibitem{LP_2008}
A.Lanzon,  and I.R.Petersen, Stability robustness of a feedback interconnection of systems with negative imaginary frequency response, \textit{IEEE Trans. Autom.
Contr.}, vol. 53, no. 4, pp. 1042--1046. 

\bibitem{P_2016}
I.R.Petersen, 
Negative imaginary systems theory and applications, 
\textit{Annu. Rev. Control}, 
vol. 42, 2016,
Pages 309--318. 

\bibitem{P_1964}
B.T.Polyak, Some methods of speeding up the convergence of iteration methods,
\emph{USSR Comp. Math. Math. Phys.}, vol. 4, no. 5, 1964, pp.  1--17.

\bibitem{PW_2025}
Y.Polyanskiy, and Y.Wu, 
\textit{Information Theory}, Cambridge University Press, 2025. 

\bibitem{RS_1978}
M.Reed,  and B.Simon, \textit{Methods of Modern Mathematical Physics. IV: Analysis of Operators},
Academic Press, New York, 1978.

\bibitem{R_1961}
A.Renyi, On measures of entropy and information, Proc. 4th Berkeley Sympos. Math. Statist. Prob., I, 1961, pp. 547--561.


\bibitem{R_1996}
H.Risken,
\emph{The Fokker-Planck Equation: Methods of Solution and Applications}, 2nd Ed., Springer, Berlin, 1996. 
\bibitem{R_1974}
F.N.H.Robinson,
\emph{Noise and Fluctuations in Electronic Devices and Circuits},
Clarendon Press, Oxford, 1974.


\bibitem{S_1994}
J.J.Sakurai, 
\textit{Modern Quantum Mechanics},
 Addison-Wesley, Reading, Mass., 1994.

\bibitem{S_1996}
A.N.Shiryaev, {\it Probability}, 2nd Ed., Springer, New York, 1996.

\bibitem{Soize_1994}
C.Soize, \textit{The Fokker-Planck Equation for Stochastic Dynamical Systems
and Its Explicit Steady State Solutions}, Series on Advances in Mathematics for
Applied Sciences, vol.  17, World Scientific, Singapore, 1994. 
 
\bibitem{S_2008}
D.W.Stroock,
\textit{Partial Differential Equations for Probabilists},
Cambridge University Press, Cambridge, 2008.

\bibitem{T_2002}
D.Talay,
Stochastic Hamiltonian systems:
exponential convergence to the
invariant measure, and discretization
by the implicit Euler scheme, \emph{Markov Processes Relat. Fields}. vol.  8, 2002, pp. 1--36.

\bibitem{T_2010}
M.E.Tuckerman, \emph{Statistical Mechanics: Theory and
Molecular Simulation}, Oxford University Press, New York, 2010.



\bibitem{UPS_2022}
V.Ugrinovskii, I.R.Petersen,  and I.Shames, Global convergence and asymptotic optimality of the heavy ball method for a class of nonconvex optimization problems, \textit{IEEE Control Syst. Lett.}, vol. 6, 2022, pp. 2449--2454. 

\bibitem{VJ_2014}
A. van der Schaft,  and D.Jeltsema, \textit{Port-Hamiltonian Systems Theory: An Introductory
Overview},  Foundations and Trends in Systems and Control, vol. 1, no. 2--3,
pp. 173--378, 2014.

\bibitem{V_2002}
V.S.Vladimirov, {\it Methods of the Theory of Generalized Functions}, CRC Press, London, 2002. 


\bibitem{VP_2018_ANZCC}
I.G.Vladimirov,  and I.R.Petersen, Dissipative linear stochastic Hamiltonian systems, 2018 Australian \& New Zealand Control Conference (ANZCC), Melbourne, VIC, Australia, 2018, pp. 227--232. 

\bibitem{VP_2020_ANZCC}
I.G.Vladimirov,  and I.R.Petersen, Mean square optimal control by interconnection for linear stochastic Hamiltonian systems, 2020 Australian and New Zealand Control Conference (ANZCC), Gold Coast, QLD, Australia, 2020, pp. 24--29. 



\bibitem{V_2023_AHP}
I.G.Vladimirov, 
Position-momentum conditioning, relative entropy decomposition and convergence to equilibrium in stochastic Hamiltonian systems, preprint: {\tt 		arXiv:2312.09475 [math-ph]}, 15 December 2023.




\bibitem{WPUS_2024}
A.X.Wu, I.R.Petersen, V.Ugrinovskii, and I.Shames, A generalized accelerated gradient optimization method,  2024 American Control Conference (ACC), Toronto, Ontario, Canada, 2024, pp. 1904--1908. 


\bibitem{Z_2001}
R.Zwanzig, \textit{Nonequilibrium Statistical Mechanics}, Oxford University Press, New
York, 2001.


\end{thebibliography}
\end{document}